\documentclass[a4paper,10pt]{amsart}
\usepackage{amssymb,amsfonts,amsthm}
\usepackage{amstext}
\usepackage{comment}
\usepackage{amsmath}
\usepackage{mathtools}
\usepackage{color}
\usepackage[active]{srcltx}
\usepackage{ dsfont }
\usepackage[latin1]{inputenc}
\usepackage{tensor}
\usepackage{graphicx}

\usepackage{chngcntr} 
\usepackage{mathrsfs} 
\usepackage{hyperref}  

\DeclareMathOperator{\shp}{sh}

\theoremstyle{plain}
\newtheorem{theorem}{Theorem}[section]
\newtheorem{lemma}[theorem]{Lemma}
\newtheorem{remark}[theorem]{Remark}
\newtheorem{definition}[theorem]{Definition}

\newtheorem{corollary}[theorem]{Corollary}

\begin{document}
	
	\title{On symmetric partition lattices and probability interactions}

	\begin{abstract}
	The M\"obius function of a partition lattice can be used to define interactions of probability measures. We call a lower set a symmetric partition lattice when it is invariant under permutations of the variables. We characterize these lattices through integer partitions and obtain a recursion for their M\"obius coefficients at the top element. By constructing probability measures whose partition products are linearly independent, we show that permutation invariance of an interaction is equivalent to that of its underlying lower set, provided the coordinate spaces are sufficiently large. We also characterize the partitions on a lower set that force an interaction to vanish for every probability measure, and construct indecomposable distributions with zero interaction. We define the order of a symmetric partition lattice and relate it to vanishing marginals. We study generalized Lancaster and define the Streitberg  and size-limited partition lattices, obtaining explicit M\"obius coefficients for them. Finally, we express interaction equations in terms of characteristic functions. For Gaussian distributions, we show that an interaction vanishes exactly when the distribution factorizes according to a nontrivial partition in the underlying lower set. We also give a nonvanishing criterion for radial characteristic functions generated by Bernstein functions.
	\end{abstract}

	\keywords{Partition lattices; Probability interactions; M\"obius inversion; Characteristic functions; Identifiability}

	\subjclass[2020]{62H20; 06A07; 60E10; 11A25; 62F01}
	
	\author{Jean Carlo Guella}
	\address{	Universidade Estadual de Mato Grosso do Sul, Nova Andradina, Brazil}
	\email{jean.guella@uems.br} 	
	\maketitle

	\tableofcontents

	\section{Introduction}
	
	Independence is a central structural assumption in probability and statistics. For a probability measure $P$ on a product space $\mathds{X}_{n}:=X_{1}\times\ldots\times X_{n}$, joint independence is expressed by the factorization $P=P_{1}\times\ldots\times P_{n}$. For two variables, this identity underlies dependence measures such as distance covariance \cite{Szekely2007,Szekely2009,Bakirov2006,Janson2021} and the Hilbert-Schmidt Independence Criterion (HSIC) \cite{Gretton2005,Gretton2008}. For more than two variables, dHSIC tests joint independence \cite{Pfister2018}.
	
	However, for $n\geq3$, the failure of joint independence does not by itself describe how the variables interact. A distribution may factorize according to a nontrivial partition of the coordinates without being jointly independent. Probability interactions, such as those of \cite{lancaster1969chi} and \cite{streitberg1990lancaster}, provide a way to study these factorizations. The Streitberg interaction vanishes whenever $P=P_{\pi}$ for a partition $\pi\neq\hat{\mathbf{1}}_{n}$, whereas the Lancaster interaction vanishes when such a partition contains a singleton block (that is, if at least  one of the variables is independent from all the others). These interactions have been studied using kernel methods \cite{NIPS2013_076a0c97} and distance multivariance \cite{Boettcher2018,Boettcher2019,Chakraborty2019}. More recently, partition lattices have provided a framework for studying higher-order interactions and associated statistical measures \cite{NEURIPS2023_74f11936,pmlr-v258-liu25f}.
	
	A natural requirement is permutation invariance: the interaction should not depend on the labeling of the variables. We define symmetric partition lattices as lower sets that are invariant under permutations and characterize them in terms of integer partitions. This characterization also gives a recursion for their M\"obius coefficients  which depend only on the size of its blocks.
	
	We first examine how the interaction determines its underlying lower set. We construct probability measures for which  its partitions are linearly independent, thus an identifiable family. Using these measures, we prove that permutation invariance of the interaction is equivalent to invariance of the lower set, provided each coordinate space contains at least $n$ points. We also show that $T[P_{\sigma}]=0$ for every probability measure $P$ precisely when $\sigma\in T\setminus\{\hat{\mathbf{1}}_{n}\}$. In contrast, we construct indecomposable distributions with zero interaction whenever $n\geq3$ and the lower set contains all partitions with $n-1$ blocks, generalizing the result of  \cite{NIPS2013_076a0c97}  proved for the Lancaster interaction of three variables.
	
	We also define the order of a symmetric partition lattice and relate it to the marginals that vanish for every associated interaction. Together with a representation of products of signed measures, this gives an extension of results in \cite{Guella2025a} that  symmetric partition lattices of the same order define the same nonnegativity conditions for the radial kernels  in all Euclidean dimensions.
	
	We study and develop three families of examples: generalized Lancaster lattices $\Lambda_{k}^{n}$, generalized Streitberg lattices $\Sigma_{k}^{n}$, and size-limited lattices $J_{k}^{n}$. We give a lattice interpretation of the generalized Lancaster interaction introduced in \cite{Guella2025a}, define and  express the M\"obius coefficients of $\Sigma_{k}^{n}$ using Stirling numbers, and obtain generating functions and marginalization formulas for the size-limited family with $k=3$.
	
	Finally, the Fourier transform converts the condition $T[P]=0$ into a functional equation for the characteristic function of $P$. For Gaussian distributions, we show that the interaction vanishes if and only if the distribution factorizes according to a nontrivial partition in $T$. We also obtain a sufficient condition for nonvanishing and apply it to radial characteristic functions generated by Bernstein functions.
	
	The paper is organized as follows. Section \ref{posetsandlattices} reviews the necessary facts about posets and partition lattices. Section \ref{symmetricpartitionlattices} develops the general results on identifiability, symmetry, and order. Sections \ref{GeneralizedLancasterinteraction}, \ref{GeneralizedStreitberginteraction}, and \ref{sizelimited} study the three families of examples. Section \ref{MultivariateGaussian} treats characteristic functions, radial functions, and Gaussian distributions.
	
	\section{Posets and lattices}\label{posetsandlattices}
	
	A finite partially ordered set, also called a poset, is a finite set $\mathcal{X}$ equipped with a partial order $\leq$.
	For the posets considered here, we assume that each has a maximum and a minimum element, which are denoted by $\hat{\mathbf{1}}$ and $\hat{\mathbf{0}}$. 
	
	If, in a poset $\mathcal{X}$, any two elements $x,y$ have a unique supremum $x\vee y$ (called the join) and a unique infimum $x\wedge y$ (called the meet), the poset is called a lattice.
	
	On a lattice $\mathcal{X}$, let  $\zeta:\mathcal{X}\times \mathcal{X}\to\mathbb{Z}$ be the function defined as
	$\zeta(x, y) = 1$ if $x\leq y$ and 0 otherwise,  called the zeta function of $\mathcal{X}$. This function has an inverse (in the matrix sense), known as the \emph{M\"obius function}   $\mu:\mathcal{X}\times \mathcal{X}\to\mathbb{Z}$, and it can be proved that it satisfies (see  Chapter 3 of \cite{Stanley_2011})
	\begin{equation}\label{posetdefn}
		\mu(x,x)=1,\qquad \mu(x,y)= -\sum_{x\leq z<y}\mu(x,z)\quad  \text{ when } x <y,
	\end{equation}
	and  $\mu(x,y)=0$ otherwise. Note that if $\hat{\mathbf{0}}\neq \hat{\mathbf{1}}$, then
	\begin{equation}\label{soma01}
		\sum_{x \in \mathcal{X}}\mu(x,\hat{\mathbf{1}} )= \sum_{x \in \mathcal{X}}\zeta(\hat{\mathbf{0}},x)\mu(x,\hat{\mathbf{1}} )= \delta_{\hat{\mathbf{0}},\hat{\mathbf{1}}}=0,
	\end{equation}
	more generally, if $ \pi \neq \hat{\mathbf{1}}$, then
	\begin{equation}\label{somax1}
		\sum_{x \geq \pi }\mu(x,\hat{\mathbf{1}} )= \sum_{x \in \mathcal{X}}\zeta(\pi,x)\mu(x,\hat{\mathbf{1}} )= \delta_{\pi,\hat{\mathbf{1}}}=0.
	\end{equation}
	
	\textbf{Streitberg lattice:} A partition $\pi$ of $\{1,\ldots,n\}$ is an unordered collection of nonempty disjoint subsets $(F_{1},\ldots,F_{\ell})$ whose union is the entire set.  In order to simplify the notation, throughout the text we always assume that $|F_{1}| \geq \ldots \geq |F_{\ell}|$, when needed, a different enumeration is stated explicitly.  We use $|\pi|=\ell$ for the number of blocks and $\Pi_{n}$ for the set of these partitions. The minimum and maximum elements are $\hat{\mathbf0}_{n}:=(\{1\},\ldots,\{n\})$ and $\hat{\mathbf1}_{n}:=\{1,\ldots,n\}$.
	
	We can define a lattice structure on $\Pi_{n}$ as follows. For $\pi=(F_{1}, \ldots, F_{|\pi|}),\sigma=(G_{1}, \ldots, G_{|\sigma|}) \in\Pi_{n}$ we write $\pi\leq \sigma$ if for every $1\leq i \leq |\pi|$ we have $F_{i} \subseteq G_{j}$ for some $1 \leq j \leq |\sigma|$. In particular, the \emph{meet} $\pi\wedge\sigma$ is the partition whose blocks are the nonempty intersections $B\cap C$ with $B\in\pi$ and $C\in\sigma$, and the \emph{join} $\pi\vee\sigma$ is the finest partition that is coarser than both, obtained by taking the transitive closure of the relation ``two elements lie in the same block of $\pi$ or of $\sigma$''.  The M\"obius function of $\Pi_{n}$ is
	\[
	\mu(\sigma,\pi)=
	\begin{cases}
		\prod_{B\in\pi} (-1)^{m_{B}-1}(m_{B}-1)!,
		& \text{if } \sigma \leq \pi,\\
		0, & \text{otherwise},
	\end{cases}
	\]
	where $m_{B}:=\#\{A\in\sigma:A\subseteq B\}$. To use this lattice in probability, consider the Cartesian product $\mathds X_{n}:=\prod_{i=1}^{n}X_{i}$. Let $\mathcal M(\mathds X_{n})$ denote the vector space of discrete signed measures on $\mathds{X}_{n}$. For a probability measure $P\in\mathcal M(\mathds X_{n})$ and $\pi=(F_{1},\ldots,F_{\ell})\in\Pi_{n}$, define $P_{\pi}:=\bigtimes_{i=1}^{\ell}P_{F_{i}}$, where $P_{F}$ is the marginal on $\mathds X_{F}:=\prod_{i\in F}X_{i}$.  A probability measure $P$ is called decomposable if $P=P_{\pi}$ for some $\pi\in\Pi_{n}\setminus\{\hat{\mathbf1}_{n}\}$.
	
	The lattice and the M\"obius function of $\Pi_{n}$ satisfy the following probabilistic characterization obtained in   \cite{streitberg1990lancaster}.
	
	\begin{theorem}\label{streit}The collection of real numbers $a_{\pi}:= (-1)^{|\pi|-1}(|\pi|-1)!$, indexed over the partitions of $\{1, \ldots, n\}$, is the only one that satisfies the following conditions on arbitrary product spaces:
		\begin{enumerate}
			\item[$(i)$] $a_{\hat{\mathbf{1}}_{n}}=1$
			\item[$(ii)$] For any decomposable probability $P$ defined on the Cartesian product $\mathds{X}_{n}$ the measure
			\[
			\Sigma[P] := \sum_{\pi \in \Pi_{n}} a_{\pi} P_{\pi}
			\]
			is the zero measure.
			\item[$(iii)$] The operator $\Sigma$ is invariant under permutations of the coordinates: if   $\tau:\{1, \ldots, n\} \to\{1, \ldots, n\}$ is a bijection, then
			\[
			\Sigma[P^{\tau}]= [\Sigma[P]]^{\tau}
			\]
			where for a signed measure $\eta$ on $\mathds{X}_{n}$ the measure $\eta^{\tau}$ is defined on $\prod_{i=1}^{n}X_{\tau(i)}$ for measurable subsets $A_{j} \subseteq X_{j}$ as $\eta^{\tau}(\prod_{i=1}^{n}A_{\tau(i)} ):=\eta (\prod_{i=1}^{n}A_{i} ) $.
		\end{enumerate}
		
	\end{theorem}
	
	As the previous theorem already suggests, the values of $\mu(\sigma, \hat{\mathbf{1}}_{n})$ are the most important in practice.  It may happen that  $\Sigma[P]=0$ even if $P$ is not decomposable, see Lemma \ref{counter}  for an example.
	
	Relation $(iii)$ in Theorem \ref{streit} is the core idea for the definition of symmetric partition lattices obtained in Theorem \ref{charsymmetricpartitionposet}. Also, relation $(ii)$ is adapted to the context of lower sets in Lemma \ref{iducedzero}.
	
	\textbf{Lancaster lattice:} Another well-known lattice, introduced in \cite{lancaster1969chi}, is $\Lambda_{n}^{n}\subseteq\Pi_{n}$, where $\pi\in\Lambda_{n}^{n}$ if and only if it contains a singleton block or is the maximum element $\hat{\mathbf1}_{n}$. The order is inherited from $\Pi_{n}$, so $(\sigma\wedge\pi)_{\Lambda_{n}^{n}}=(\sigma\wedge\pi)_{\Pi_{n}}$. The join $(\sigma\vee\pi)_{\Lambda_{n}^{n}}$ equals $(\sigma\vee\pi)_{\Pi_{n}}$ when the latter belongs to $\Lambda_{n}^{n}$, otherwise, it equals $\hat{\mathbf1}_{n}$. For $\pi=(F_{1},\ldots,F_{\ell})\in\Lambda_{n}^{n}$, its M\"obius function satisfies
	\[
	\mu(\pi, \hat{\mathbf{1}}_{n})=
	\begin{cases}
		1, & \text{ if } \pi= \hat{\mathbf{1}}_{n}\\
		0,& \text{ if } |F_{2}| \geq 2,\\
		(-1)^{n-|F_{1}|}, & \text{ if } |F_{2}|=1 \text{ and } |F_{1}|\geq 2,\\
		(-1)^{n-1}(n-1) & \text{ if } \pi= \hat{\mathbf{0}}_{n}.
	\end{cases}
	\]
	As with the Streitberg lattice, we can define the probability interaction of a probability measure $P$ on $\mathds{X}_{n}$ as
	\[
	\Lambda_{n}^{n}[P]:= \sum_{\pi \in \Lambda_{n}^{n}}\mu(\pi, \hat{\mathbf{1}}_{n})P_{\pi} = \sum_{\ell=0}^{n}(-1)^{n-\ell} \sum_{|F|=\ell}P_{F}\times \left [\bigtimes_{i \in F^{c}}P_{i} \right ],
	\]
	because $P_{\emptyset}:=1$ and
	\[
	\sum_{\ell=0}^{1}(-1)^{n-\ell} \sum_{|F|=\ell}P_{F}\times \left [\bigtimes_{i \in F^{c}}P_{i} \right ]= (-1)^{n-1}(n-1)\bigtimes_{i =1}^{n}P_{i}.
	\]
	
	\section{Symmetric partition lattices}\label{symmetricpartitionlattices}
	
	In this section we aim to  formally define what a probability interaction based on the concept of a lower set of $\Pi_{n}$ is and the concept of a symmetric partition lattice of $\Pi_{n}$. 
	
	First, we note that relation $(ii)$ in Theorem \ref{streit} is part of a broader result in lattice theory which is also a generalization of Weisner's Theorem, see Corollary 3.9.3 in \cite{Stanley_2011}.
	
	\begin{theorem} \label{generalized_orthogonality}
		For any function $f : \mathcal{X} \to \mathbb{R}$ and any element $\sigma \in \mathcal{X} \setminus \{\hat{\mathbf{1}}\}$, the following   identity holds:
		\[
		\sum_{\pi \in \mathcal{X}} \mu(\pi, \hat{\mathbf{1}}) f(\pi \wedge \sigma) = 0.
		\]
	\end{theorem}
	
	Motivated by the usual notion of a lower set \cite[Chapter 3]{Stanley_2011}, we use the following convention, in which the top element is adjoined.
	\begin{definition}
		A subset $S\subseteq\mathcal X$ is called a lower set in this paper if it contains $\hat{\mathbf0}$ and $\hat{\mathbf1}$ and satisfies the following condition: for every $w\in S\setminus\{\hat{\mathbf1}\}$ and $z\in\mathcal X$, if $z\leq w$, then $z\in S$.
	\end{definition}
	
	For a nonempty subset $F \subset \mathcal{X}\setminus{\{\hat{\mathbf{1}}\}}$ we may define the lattice generated by $F$ as
	\begin{equation}\label{latticegenerated}
		S_{F}:=\{ x \in \mathcal{X}, \quad  x \leq y \text{ for some } y \in F\} \cup \{\hat{\mathbf{1}} \},
	\end{equation}
	which is  a lower set. Conversely, for a lower set  $S \subset \mathcal{X}$, we define its set of maximal  elements by
	\[
	\max(S):= \{ x \in S\setminus{ \{\hat{\mathbf{1}}\}}, \quad \text{ if } y \in S \text{ and }  y> x \text{ then } y=\hat{\mathbf{1}}  \},
	\]
	and clearly we have that $S = S_{\max(S)}$.	If   $S \subseteq \mathcal{X}$ is a lower set,  we can also define a concept of join on $S$ based on that of $\mathcal{X}$, by setting
	\[
	x\vee_{S} y =
	\begin{cases}
		x\vee_{\mathcal{X}} y & \text{ if } 	x\vee_{\mathcal{X}} y \in S. \\
		\hat{\mathbf{1}} & \text{ otherwise}.
	\end{cases}
	\]
	It is a simple exercise to show that with  this definition of join, the poset $S$ is a lattice. In particular, Theorem \ref{generalized_orthogonality} holds when restricted to $S$, in the sense that for any function $f : S \to \mathbb{R}$ and any element $\sigma \in S \setminus \{\hat{\mathbf{1}}\}$, the following   identity holds
	\begin{equation}\label{generalized_orthogonalityS}
		\sum_{\pi \in S} \mu_{S}(\pi, \hat{\mathbf{1}}) f(\pi \wedge \sigma) = 0.
	\end{equation}
	
	A useful technique that leads to proofs by induction on a finite lattice $\mathcal{X}$ is the concept of level sets.  We define $I_{1} := \{\hat{\mathbf{1}}\}$, and inductively define the nonempty sets:
	\[
	I_{k+1} := \left\{ x \in \mathcal{X} \setminus \bigcup_{i=1}^{k} I_{i} \;\middle|\; \text{if } y \in \mathcal{X} \setminus \bigcup_{i=1}^{k} I_{i} \text{ and } y \geq x \text{ then } y = x \right\}.
	\]
	Because $\mathcal{X}$ is finite, this recursion terminates at an integer $h$, where $I_{h} = \{\hat{\mathbf{0}}\}$. This value $h$ is defined as the  height of the lattice $\mathcal{X}$ (the maximum cardinality of a chain). Note that the level sets are disjoint and if $x \in I_{i}$ and $y \in I_{j}$ with $ x < y $  then $j< i$. 
	
	\begin{remark} \label{maiorelementosubposet}
		It can be proved using induction on the level sets  that if  $\mathcal{X}$ is a finite lattice  and  $S \subseteq \mathcal{X}$ is a lower set, then the M\"obius function of $S$ is identical to the M\"obius function of $\mathcal{X}$ strictly below the top element:
		\[
		\mu_{S}(x, y) = \mu_{\mathcal{X}}(x, y),  \quad  \text{for any } x,y \in S \text{ with } y < \hat{\mathbf{1}}.
		\]
	\end{remark}
	
	From now on, our focus is on lower sets of the lattice $\Pi_{n}$.  We use the notation $\hat{\mathbf{1}}_{n}$ for the one-block partition $\{1,\ldots, n\}$ (the largest element of $\Pi_{n}$) to simplify some expressions. If $T \subset \Pi_{n}$ is a lower set  and  $P \in \mathcal{M}(\mathds{X}_{n})$ is a probability measure, we define the interaction $T[P]$ by the signed measure
	\begin{equation}\label{interaction}
		T[P]  :=  \sum_{\pi \in T}\mu(\pi, \hat{\mathbf{1}}_{n})P_{\pi}.
	\end{equation}
	
	\begin{lemma}\label{iducedzero}Let $T \subset \Pi_{n}$ be a lower set  and  $P \in \mathcal{M}(\mathds{X}_{n})$ be a probability.  If $\sigma \in T \setminus\{\hat{\mathbf{1}}_{n}\}$   then $T[P_{\sigma}]=0$.	\end{lemma}
	
	\begin{proof}
		This follows from Equation \ref{generalized_orthogonalityS}:
		\[
		T[P_{\sigma}]=  \sum_{\pi \in T}\mu(\pi, \hat{\mathbf{1}}_{n})(P_{\sigma})_{\pi} =  \sum_{\pi \in T}\mu(\pi, \hat{\mathbf{1}}_{n})P_{\sigma \wedge \pi} =0.
		\]\end{proof}
	
	A family of probability measures generates identifiable finite mixtures if and only if it is linearly independent. For example, the family of distinct nondegenerate Gaussian measures $N(\omega,\Sigma)$ is linearly independent \cite{10.1214/aoms/1177698520}.
	
	For the purposes of this paper, we call a probability measure $P\in\mathcal M(\mathds X_{n})$ marginally identifiable if the indexed family $(P_{\pi})_{\pi\in\Pi_{n}}$ is linearly independent. The next lemma shows that such measures exist whenever each coordinate space contains at least $n$ points. We then use them to prove a converse of Lemma \ref{iducedzero}.
	
	\begin{lemma}\label{marginalidentifiability}
		Let $p\geq n$ and $x_{i}=(x_{1}^{i},\ldots,x_{n}^{i})\in\mathds X_{n}$, $1\leq i\leq p$, be such that, for every coordinate $j$, the points $x_{j}^{1},\ldots,x_{j}^{p}$ are pairwise distinct. Then $P:=p^{-1}\sum_{i=1}^{p}\delta_{x_{i}}$ is marginally identifiable.
	\end{lemma}
	
	\begin{proof} Let $b_{\pi}\in \mathbb{R}$ be such that $\sum_{\pi \in \Pi_{n}}b_{\pi} P_{\pi}=0$. For $\alpha\in\{1,\ldots,p\}^{n}$, write $x_{\alpha}=(x_{1}^{\alpha_{1}},\ldots,x_{n}^{\alpha_{n}})$. First note that this sum is a linear combination of $\delta_{x_{\alpha}}$ for $\alpha \in \{1, \ldots, p\}^{n}$, which by the hypothesis are all distinct. To obtain the coefficients that multiply the Dirac measure $\delta_{x_{\alpha}}$, define the partition $\pi_{\alpha} \in \Pi_{n}$, by the property that $i, j \in \{1, \ldots,n\}$ are in the same block if and only if $\alpha_{i}=\alpha_{j}$.
		Since, for $\pi=(F_{1}, \ldots, F_{|\pi|}) \in \Pi_{n}$ we have that
		\[
		P_{\pi} =\frac{1}{p^{|\pi|}} \bigtimes_{j=1}^{|\pi|} \left (\sum_{i=1}^{p}\delta_{(x_{i})_{F_{j}}}\right )= \frac{1}{p^{|\pi|}}\sum_{\beta \in \{1, \ldots, p\}^{n}  \mid  \pi \leq \pi_{\beta} }\delta_{x_{\beta}}.
		\]
		Consequently, if $\sum_{\pi\in\Pi_{n}}b_{\pi} P_{\pi}=0$, then
		\[
		0=\sum_{\pi \in \Pi_{n}}b_{\pi} P_{\pi}= \sum_{\alpha\in \{1, \ldots, p\}^{n} }\left ( \sum_{\pi \leq \pi_{\alpha}}b_{\pi}\frac{1}{p^{|\pi|}} \right )\delta_{x_{\alpha}},
		\]
		thus, for any $\alpha \in \{1, \ldots, p\}^{n}$
		\[
		0=\left ( \sum_{\pi \leq \pi_{\alpha}}b_{\pi}\frac{1}{p^{|\pi|}} \right )= \sum_{\sigma \in \Pi_{n}} \frac{b_{\sigma}}{p^{|\sigma|}}\zeta(\sigma,\pi_{\alpha}).
		\]
		Since $p\geq n$,  the set $\{ \pi_{\alpha}, \quad \alpha \in \{1, \ldots, p\}^{n} \}$ encompasses all partitions in $\Pi_{n}$. Since the zeta function is invertible, we obtain that all $b_{\pi}=0$.		\end{proof}
	
	Lemma \ref{marginalidentifiability} still holds true if $P =\sum_{i=1}^{p}c_{i}\delta_{x_{i}}$, where $0 < c_{i} < 1$ and  $\sum_{i=1}^{p}c_{i}=1$. To present the converse of Lemma \ref{iducedzero}, we use the same marginally identifiable probability measure.
	
	\begin{corollary} \label{iducedzero2}Let $T \subset \Pi_{n}$ be a lower set and suppose that  $|X_{i}| \geq n$ for all $i$. Then a partition $\sigma \in \Pi_{n}$ satisfies $T[P_{\sigma}]=0$ for all probabilities  $P \in \mathcal{M}(\mathds{X}_{n})$ if and only if $\sigma  \in T \setminus\{\hat{\mathbf{1}}_{n}\}$.
	\end{corollary}
	
	\begin{proof}
		The reverse implication is Lemma \ref{iducedzero}. For the forward implication, choose a marginally identifiable probability measure $P$ and write
		\[
		0=T[P_{\sigma}]=\sum_{\pi\in T}\mu(\pi,\hat{\mathbf1}_{n})P_{\sigma\wedge\pi}.
		\]
		If $\sigma\notin T$, then $\sigma\wedge\pi<\sigma$ for every proper $\pi\in T$, since otherwise the lower-set property would imply $\sigma\in T$. Thus the only coefficient that multiplies $P_{\sigma}$ is $\mu(\hat{\mathbf1}_{n},\hat{\mathbf1}_{n})=1$, contradicting linear independence.
	\end{proof}
	
	However, as shown in \cite{NIPS2013_076a0c97} for the lattice $\Pi_{3}$, the fact that a probability  $P \in \mathcal{M}(\mathds{X}_{n})$ satisfies $T[P]=0$ does not directly  imply that $P=P_{\pi}$ for some $\pi \in T \setminus\{\hat{\mathbf{1}}_{n}\}$. In the next lemma we provide a general construction exhibiting the same phenomenon. In its proof, we repeatedly use the fact that for $p \in \mathbb{N}$ it holds that   $\sum_{\alpha \in \{1,2\}^{p}} (-1)^{\alpha_{j}}=0$ for any $1\leq j \leq p$.
	
	\begin{lemma}\label{counter} Let $n\geq3$ and let $T \subset \Pi_{n}$ be a lower set such that  $\{\pi \in \Pi_{n}, \quad |\pi|=n-1\} \subset T$ and suppose that $|X_{i}| \geq 2$ for all $i$. Then there exists a probability   $P \in \mathcal{M}(\mathds{X}_{n})$ such that $T[P]=0$ but $P \neq P_{\pi}$ for all $\pi \in \Pi_{n}\setminus\{\hat{\mathbf{1}}_{n}\}$.
	\end{lemma}
	
	\begin{proof} Let $x_{i}^{1}\neq x_{i}^{2}$ be  points in $X_{i}$ for $1\leq i \leq n$ and  write $x_{\alpha}=(x_{1}^{\alpha_{1}},\ldots,x_{n}^{\alpha_{n}})$ for $\alpha\in\{1,2\}^{n}$. For $0<\varepsilon<1/(n-1)$, define the measure
		\[
		P=\frac{1}{2^{n}}\sum_{\alpha \in \{1,2\}^{n}}\left [1 +  \varepsilon (-1)^{\alpha_{1}}\left (  \sum_{j=2}^{n}(-1)^{\alpha_{j}}\right ) \right ]\delta_{x_{\alpha}}.
		\]
		Then $P$ is a probability because
		\[
		P(\{x_{\alpha}\})= \frac{1}{2^{n}}\left [1 +  \varepsilon (-1)^{\alpha_{1}}\left (  \sum_{j=2}^{n}(-1)^{\alpha_{j}}\right ) \right ] \geq  \frac{1}{2^{n}}(1 - \varepsilon(n-1)) >0,
		\]
		\[
		\sum_{\alpha \in \{1,2\}^{n}}P(\{x_{\alpha}\})= \sum_{\alpha \in \{1,2\}^{n}}\frac{1}{2^{n}}+  	\frac{\varepsilon}{2^{n}} \sum_{j=2}^{n}\left (\sum_{\alpha \in \{1,2\}^{n}}  (-1)^{\alpha_{1}+\alpha_{j}}\right )=1+0=1.
		\]
		Now let $F\subseteq\{1,\ldots,n\}$ be nonempty. Its marginal $P_{F}$ depends on whether $1\in F$, as follows:
		\[
		P_{F}(\{x_{\alpha_{F}}\})=
		\begin{cases}
			2^{-|F|},
			& \text{ if }  1\notin F,\\
			\left (1 + \varepsilon (-1)^{\alpha_{1}}\sum_{j \in F\setminus{\{1\}}}(-1)^{\alpha_{j}}\right )2^{-|F|}  , & \text{ if } 1 \in F \text{ and } |F|>1
		\end{cases}
		\]
		and $P_{\{1\}}= (\delta_{x_{1}^{1}}+\delta_{x_{1}^{2}})/2 $. Now let $\pi\in\Pi_{n}\setminus\{\hat{\mathbf1}_{n}\}$. Then there must exist a $j \in \{2, \ldots, n\}$ for which $1$ and $j$ are not in the same block of $\pi$. In particular, if $P = P_{\pi}$ we would have the marginal equality $P_{\{1,j\}}=(P_{\pi})_{\{1,j\}} = P_{\{1\}}\times P_{\{j\}}$ which cannot occur by the description of $P_{F}$.\\
		For $\pi\in\Pi_{n}$, including the top partition, let $B_{\pi}^{1}$ be the block containing $1$. For $2\leq j\leq n$, let $\sigma_{j}$ be the partition with block $\{1,j\}$ and all other blocks singletons, thus $|\sigma_{j}|=n-1$ and  $\sigma_{j} \in T$. Then
		\[
		P_{\pi}(\{x_{\alpha}\})= (1 + \varepsilon (-1)^{\alpha_{1}}\sum_{j \in B_{\pi}^{1}\setminus{\{1\}}}(-1)^{\alpha_{j}})2^{-n}= (1 + \varepsilon \sum_{\sigma_{j}\leq \pi }(-1)^{\alpha_{1} + \alpha_{j}})2^{-n}.
		\]
		To conclude that $T[P]$ is the zero measure we use Equations \ref{soma01}	and \ref{somax1} to get
		\[
		T[P](\{x_{\alpha}\})=\sum_{\pi \in T}\mu(\pi,\hat{\mathbf 1}_{n})P_{\pi}(\{x_{\alpha}\})=2^{-n}\sum_{\pi \in T}\mu(\pi,\hat{\mathbf 1}_{n}) + \varepsilon2^{-n}\sum_{j=2}^{n}(-1)^{\alpha_{1} + \alpha_{j}} \sum_{\substack{\pi\in T\\\pi\geq\sigma_{j}}}\mu(\pi,\hat{\mathbf 1}_{n})=0.
		\]
	\end{proof}

	To define the  concept of a symmetric partition lattice we need some additional terminology.
	
	\begin{definition}\label{order[n]} For $n \in \mathbb{N}$, we define  its integer-partition poset
		\[
		\operatorname{Par}(n):= \left\{ \alpha \in \mathbb{N}^{\ell} \mid 1 \leq \ell \leq n, \quad \alpha_{1} \geq \ldots \geq \alpha_{\ell} \geq 1 \text{ and } \sum_{i=1}^{\ell} \alpha_{i}=n \right\}.
		\]
		For $\alpha\in\operatorname{Par}(n)$, let $\ell(\alpha)$ be its number of parts. Define $\alpha\preceq\beta$ if the indices of the parts of $\alpha$ can be partitioned into blocks $(G_{1},\ldots,G_{\ell(\beta)})$ such that $\sum_{j\in G_{i}}\alpha_{j}=\beta_{i}$. The maximal element is denoted by $(n)$.
		
		For $\pi\in\Pi_{n}$, we define its shape by $\shp(\pi)$, as  the integer partition obtained by listing its block sizes in decreasing order. For $\alpha\preceq\beta$, define $m(\alpha,\beta)$ as the number of coarsenings of shape $\beta$ of a fixed partition of shape $\alpha$, and set $m(\alpha,\beta)=0$ otherwise. This number does not depend on the chosen partition. Equivalently, it counts the groupings of the parts of $\alpha$, with their indices distinguished, that yield $\beta$; merely reordering the groups does not give a new grouping. For example, $m((2,1,1),(3,1))=2$, corresponding to $(\{1,2\},\{3\})$ and $(\{1,3\},\{2\})$.
		
		We also define for $\alpha=(\alpha_{1}, \ldots, \alpha_{\ell})   \in \operatorname{Par}(n)$ the largest part $\alpha_{1}$, that is, the largest entry of $\alpha$.
	\end{definition}
	
	We are not interested in the behavior of the M\"obius function on lower sets of the poset  $\operatorname{Par}(n)$. However, in the next theorem we show that the analysis of the class of symmetric lower sets of $\Pi_{n}$ is connected to lower sets of $\operatorname{Par}(n)$.
	
	\begin{theorem} \label{charsymmetricpartitionposet}
		Let $T \subset \Pi_{n}$ be a lower set that  is invariant under the action of the symmetric group, that is, if $\pi=(F_{1}, \ldots, F_{\ell}) \in T$, then for any bijection $\tau:\{1, \ldots, n\}\to \{1, \ldots, n\}$, the permuted partition $\tau(\pi)= (\tau(F_{1}), \ldots, \tau(F_{\ell}))$ is also an element of $T$.\\
		Then, there exists a lower set $S \subset \operatorname{Par}(n)$  such that $\pi \in T$ if and only if $\shp(\pi) \in S$. Conversely, every lower set $S\subseteq\operatorname{Par}(n)$ gives an invariant lower set by this rule. Furthermore, $\mu(\pi, \hat{\mathbf{1}}_{n})= H(\shp(\pi))$, for some function  $H: S \to \mathbb{R}$, which satisfies the recursive relations:
		\[
		H((n))=1, \quad H(\alpha) = - \sum_{\substack{\alpha \prec \beta \preceq (n) \\ \beta \in S}}H(\beta)m(\alpha, \beta).
		\]
	\end{theorem}
	
	\begin{proof} Define $S \subset \operatorname{Par}(n)$ as the set of $\alpha \in \operatorname{Par}(n)$ for which there exists a $\pi \in T$ such that $\alpha = \shp(\pi)$.\\
		Now,  we prove that if $\shp(\sigma)  \in S$, then $\sigma \in T$. Indeed,  since  $\shp(\sigma)  \in S$  there must exist a partition $\pi \in T$ for which   $\shp(\sigma) =\shp(\pi)$. Hence, if $\sigma=(G_{1}, \ldots, G_{\ell})$ and $\pi = (F_{1}, \ldots, F_{\ell})$, as we always have $|G_{i}|= |F_{i}|$, the union of bijections from $F_{i}$ to $G_{i}$ defines a permutation taking $\pi$ to $\sigma$. The invariance of $T$ gives $\sigma\in T$.\\ Now, if $\alpha\in S\setminus\{(n)\}$ and $\gamma\preceq\alpha$, choose $\pi\in T$ with shape $\alpha$ and refine its blocks to obtain a partition $\sigma\leq\pi$ of shape $\gamma$. Since $\pi$ is proper, $\sigma\in T$, so $\gamma\in S$. Thus $S$ is a lower set. Conversely, the preimage of a lower set $S$ under $\shp$ is a lower set of $\Pi_{n}$ and is invariant under permutations.\\
		We construct the function $H$ and prove its properties by induction on the level sets of $S$ using   Equation \ref{somax1}. First, note that if $\pi > \sigma $ then $\shp(\pi) \succ \shp(\sigma)$. \\
		Define $	H((n))= \mu(\hat{\mathbf{1}}_{n}, \hat{\mathbf{1}}_{n})=1$.  For $\alpha \in I_{2}$, for any $\pi \in T$ with $\shp(\pi) =\alpha$ we have that  $m(\shp(\pi),(n))=1$ and then  $\mu(\pi, \hat{\mathbf{1}}_{n})= H(\shp(\pi))=-1$.\\
		Now, suppose that we have already  constructed the function $H$ and proved its properties on $I_{1}, \ldots, I_{k}$. We prove the corresponding properties on $I_{k+1}$. Let $\alpha\in I_{k+1}$. Then for any $\pi \in T$ with $\shp(\pi) =\alpha$ we have that
		\[
		\mu(\pi, \hat{\mathbf{1}}_{n}) = -\sum_{x> \pi} \mu(x, \hat{\mathbf{1}}_{n})=  -\sum_{x> \pi} H(\shp(x))=  - \sum_{\substack{\alpha \prec \beta \preceq (n) \\ \beta \in S}}H(\beta)\#\{x \in T \mid \shp(x) = \beta \text{ and } x > \pi \},
		\]
		The conclusion follows from the fact that $\#\{x \in T \mid \shp(x) = \beta \text{ and } x > \pi \}= m(\alpha, \beta)$ for any $\beta \in S$.
	\end{proof}
	
	We emphasize that the function $H: S \to \mathbb{R}$ obtained in  Theorem \ref{charsymmetricpartitionposet} is not the M\"obius function on $S$. As an example, in $\Pi_{3}$  we have the Lancaster interaction (or Streitberg, as they are the same in this case) with M\"obius function  $H((3))=1,  H((2,1))=-1, H((1,1,1))=2$, but the M\"obius function on the poset $\operatorname{Par}(3)$ is  $\mu((3),(3))=1,  \mu((2,1),(3))=-1,  \mu((1,1,1),(3))=0$.
	
	If a lower set  $T \subset \Pi_{n}$ satisfies the invariance hypothesis of Theorem \ref{charsymmetricpartitionposet}, we say that $T$ is a symmetric partition lattice. We define the order of $T$ as
	\[
	|T|_{\infty}:= 1 + \max_{\alpha \in S \setminus{\{(n)} \}}\alpha_{1}.
	\]
	For instance, both the Lancaster and the Streitberg lattices have order $n$. A reason for the  term $+1$ in this definition is to keep the indexing consistent with \cite{Guella2025a}. We emphasize that the value of the order of $T$ can be quite different from its height. For instance, the symmetric partition lattice $J_{3}^{2n}$ defined in Section \ref{sizelimited} has a height of $n+2$ while its order is $3$ (for $n\geq2$). Every symmetric partition lattice of order at least $3$ contains all partitions with $n-1$ blocks. Thus Lemma \ref{counter} applies to each such lattice when the coordinate spaces contain at least two points.
	
	Now we prove   two results  that will lead to some important consequences for the use of symmetric partition lattices in kernel methods and related aspects. For the first result, given probability measures $P^{1}_{i}, P_{i}^{2} \in \mathcal{M}(X_{i})$  and $\alpha=(\alpha_{1}, \ldots, \alpha_{n}) \in \mathbb{N}_{2}^{n}:=\{1,2\}^{n}$, we use the notation $P_{\alpha}:=\bigtimes_{i=1}^{n} P_{i}^{\alpha_{i}}$ and $|\alpha|:=\sum_{i=1}^{n}\alpha_{i} $.

	\begin{theorem}\label{genStreitbergcartesianproduct}
		Let $T\subseteq\Pi_{n}$ be a symmetric partition lattice of order $k$, and let $\mu_{i}\in\mathcal M(X_{i})$, $1\leq i\leq n$, satisfy $\#\{i:\mu_{i}(X_{i})=0\}\geq k$. There exist $M>0$ and a probability measure $P\in\mathcal M(\mathds X_{n})$ such that $T[P]=M(-1)^{n}\bigtimes_{i=1}^{n}\mu_{i}$.
	\end{theorem}
	
	\begin{proof} We  may assume that all measures $\mu_{i}$ are nonzero, because otherwise we take $M=1$ and an arbitrary $P $ that satisfies $P=\bigtimes_{i=1}^{n}P_{i}$.\\
		Let $\mu_{i}  \in \mathcal{M}(X_{i})$, $1\leq i \leq n$,  with the restriction that  $\#\{i:\mu_{i}(X_{i})=0\}\geq k$. For convenience, we assume that $\{1, \ldots, k\} \subset \{i, \quad \mu_{i}(X_{i})=0\}$. Then, for $1\leq i \leq k$ the Hahn-Jordan decomposition of $\mu_{i}$  can be written as  $\mu_{i}= b_{i}[S_{i}^{1}  - S_{i}^{2} ]$, where $b_{i}$ is positive  and $S_{i}^{1}, S_{i}^{2}$ are probability measures on $X_{i}$. Then
		\[
		\bigtimes_{i=1}^{k}\mu_{i} = B(\bigtimes_{i=1}^{k}[S_{i}^{1} - S_{i}^{2}])= B\sum_{\alpha\in \mathbb{N}_{2}^{k}}(-1)^{k- |\alpha|}S_{\alpha}=  (-1)^{k}B\left [ \sum_{|\alpha| \in 2\mathbb{Z}}S_{\alpha} - \sum_{|\alpha| \in 2\mathbb{Z}+1}S_{\alpha}\right ],
		\]
		where $B=  \prod_{i=1}^{k}b_{i}$.  Also, for $i\geq k+1$,  let $\mu_{i}= c_{i}^{1}R_{i}^{1}  - c_{i}^{2}R_{i}^{2} $ be a  Hahn-Jordan decomposition, where $c_{i}^{1}, c_{i}^{2}$ are nonnegative and $R_{i}^{1}, R_{i}^{2}$ are probability measures on $X_{i}$. We then have
		\[
		\bigtimes_{i=k+1}^{n}\mu_{i} = \sum_{\beta\in \mathbb{N}_{2}^{n-k}}(-1)^{n-k- |\beta|}c_{\beta}R_{\beta} = (-1)^{n-k} \left [ \sum_{|\beta| \in 2\mathbb{Z}}c_{\beta}R_{\beta} - \sum_{|\beta| \in 2\mathbb{Z}+1}c_{\beta}R_{\beta}\right ],
		\]
		where $c_{\beta} := \prod_{i=k+1}^{n}c_{i}^{\beta(i-k)}$ and $R_{\beta }:= \bigtimes_{i=k+1}^{n}R_{i}^{\beta(i-k)}$. Thus
		\[
		(-1)^{n}\bigtimes_{i=1}^{n}\mu_{i}=  B\left [ \sum_{|\alpha| \in 2\mathbb{Z}}S_{\alpha} - \sum_{|\alpha| \in 2\mathbb{Z}+1}S_{\alpha}\right ]  \times \left [ \sum_{|\beta| \in 2\mathbb{Z}}c_{\beta}R_{\beta} - \sum_{|\beta|   \in 2\mathbb{Z}+1}c_{\beta}R_{\beta}\right ].
		\]
		Define the probability
		\[
		P:= \frac{1}{D}\left [\left (  \sum_{|\alpha| \in 2\mathbb{Z}}S_{\alpha} \right ) \times \left ( \sum_{|\beta| \in 2\mathbb{Z}}c_{\beta}R_{\beta} \right ) + \left ( \sum_{|\alpha| \in 2\mathbb{Z}+1}S_{\alpha} \right ) \times \left (  \sum_{|\beta| \in 2\mathbb{Z}+1}c_{\beta}R_{\beta} \right ) \right ],
		\]
		where $D = 2^{k-1}\sum_{\beta\in \mathbb{N}_{2}^{n-k}}c_{\beta}$. Then, if $i \in \{1, \ldots, k\}$
		\[
		P_{\{1,\ldots,n\}\setminus\{i\}}=\frac{2^{k-1}}{D}
		\left(\bigtimes_{\substack{j=1\\j\neq i}}^{k}\frac{S_{j}^{1}+S_{j}^{2}}{2}\right)\times\left(\sum_{\beta\in\mathbb N_{2}^{n-k}}c_{\beta} R_{\beta}\right)=\left(\bigtimes_{\substack{j=1\\j\neq i}}^{k}\frac{S_{j}^{1}+S_{j}^{2}}{2}\right)\times\left(\bigtimes_{j=k+1}^{n}\frac{c_{j}^{1}R_{j}^{1}+c_{j}^{2}R_{j}^{2}}{c_{j}^{1}+c_{j}^{2}}\right).
		\]
		From this relation we obtain that
		\[
		P_{i}= \frac{S_{i}^{1} + S_{i}^{2}}{2}, \text{ for $1\leq i \leq k$  and } P_{i}= \frac{c_{i}^{1}R_{i}^{1} + c_{i}^{2}R_{i}^{2}}{c_{i}^{1} + c_{i}^{2}}, \text{ for $i \geq k+1$}.
		\]
		In particular,  we also  obtain that  for any $F \subset \{1, \ldots, n\}$, with $|F|\leq k-1$, it holds that $P_{F}= \bigtimes_{i\in F }P_{i}$,  because there must exist an index in $\{1, \ldots, k\} \cap F^{c}$. Gathering all these relations, we conclude that for any $\pi \in T\setminus\{\hat{\mathbf{1}}_{n}\}$ it holds that $P_{\pi}= \bigtimes_{i=1 }^{n}P_{i}$. 	Thus, using $\mu(\hat{\mathbf1}_{n},\hat{\mathbf1}_{n})=1$ and Equation \ref{soma01}, we get that
		\[
		T[P]=    \sum_{\pi \in T}\mu(\pi, \hat{\mathbf{1}}_{n})P_{\pi}=P  +  \left [\bigtimes_{i=1}^{n} P_{i}  \right ]\sum_{\pi \in T\setminus\{\hat{\mathbf{1}}_{n}\} }\mu(\pi, \hat{\mathbf{1}}_{n}) =  P -\bigtimes_{i=1}^{n} P_{i},
		\]
		On the other hand,
		\[
		\bigtimes_{i=1}^{n} P_{i}=\frac{1}{2D} \left ( \sum_{\alpha \in \mathbb{N}_{2}^{k}} S_{\alpha}  \right ) \times \left (  \sum_{\beta \in \mathbb{N}_{2}^{n-k}} c_{\beta}R_{\beta}   \right ),
		\]
		Therefore, $T[P]$ is equal to 
		\[
		P -\bigtimes_{i=1}^{n} P_{i}= \frac{1}{2D}  \left [\left (  \sum_{|\alpha| \in 2\mathbb{Z}}S_{\alpha} - \sum_{|\alpha| \in 2\mathbb{Z}+1}S_{\alpha}\right ) \times \left ( \sum_{|\beta| \in 2\mathbb{Z}}c_{\beta}R_{\beta} - \sum_{|\beta| \in 2\mathbb{Z}+1}c_{\beta}R_{\beta} \right ) \right ]=\frac{(-1)^{n}}{2BD}\bigtimes_{i=1}^{n}\mu_{i}.
		\]\end{proof}
	
	For the second result, we recall the definition of the vector space
	\[
	\mathcal{M}_{k}( \mathds{X}_{n}):=\{\mu \in \mathcal{M}( \mathds{X}_{n}) : \mu (\prod_{i=1}^{n}A_{i} )=0, \text{ if } | \{i: A_{i}=X_{i}\} |\geq n-k+1 \},
	\]
	defined in  \cite{Guella2025a}.
	
	\begin{theorem}\label{genStreitbergvectorspace}Let $T \subset \Pi_{n}$ be a symmetric partition lattice of order $k$ and  $P \in \mathcal{M}(\mathds{X}_{n})$ be a probability. Then  $T[P] \in  \mathcal{M}_{k}(\mathds{X}_{n})$.\end{theorem}
	
	\begin{proof} By hypothesis, and with the notation of Theorem \ref{charsymmetricpartitionposet}, there is some $\alpha=(k-1, \alpha_{2}, \ldots, \alpha_{\ell} ) \in S$. Since $(k-1,1^{n-k+1}) \preceq \alpha $, we obtain by the hypothesis on $T$ that  $(k-1,1^{n-k+1}) \in S$.\\
		Let $F=\{1, \ldots, k-1\}$   and define the partition  $\sigma=(F, \{k\}, \{k+1\}, \ldots, \{n\})$. Hence, $\shp(\sigma)= (k-1,1^{n-k+1}) $ and then $\sigma \in T \setminus\{\hat{\mathbf{1}}_{n}\}$, so Lemma \ref{iducedzero} gives $T[P_{\sigma}]=0$.\\
		However, for any $\pi \in T$ we have that  $(P_{\pi})_{F} = (P_{\pi \wedge \sigma})_{F}$, and then
		\[
		(T[P])_{F} = \left ( \sum_{\pi \in T} \mu(\pi, \hat{\mathbf{1}}_{n})P_{\pi}  \right )_{F}=\sum_{\pi \in T} \mu(\pi, \hat{\mathbf{1}}_{n}) (P_{\pi})_{F}  =\sum_{\pi \in T} \mu(\pi, \hat{\mathbf{1}}_{n})(P_{\pi \wedge \sigma})_{F}= (T[P_{\sigma}])_{F}=0.
		\]
		By the same argument, we can prove that this equality holds for any $F \subset \{1, \ldots, n\}$ with $|F|=k-1$, Marginalizing once more gives the same conclusion for smaller $F$, which concludes the proof.
	\end{proof}
	
	The algebraic identities and properties for probability interactions extend to  finite Radon measures on Hausdorff spaces by the same arguments, see \cite{Guella2025a}.

	We conclude this section by showing an equivalence of the  concept of symmetric partition lattice with  property $(iii)$ of Theorem \ref{streit}. 
	\begin{theorem} \label{aaa}
		Let $T \subset \Pi_{n}$ be a lower set and let $X_{i}$ be sets, $1\leq i \leq n$ with $|X_{i}| \geq n$. Then the property of $T $ being invariant under the symmetric group, as in Theorem \ref{charsymmetricpartitionposet}, is equivalent to property $(iii)$ of Theorem \ref{streit} for all probability measures on $\mathds{X}_{n}$.
	\end{theorem}
	
	\begin{proof}If $T$ is invariant under the symmetric group, we have  that 
		\[
		\begin{aligned}
			(T[P^{\tau}])&= \sum_{\pi \in T}\mu(\pi,\hat{\mathbf 1}_{n})(P^{\tau})_{\pi}=  \sum_{\sigma \in T}\mu(\tau^{-1}(\sigma),\hat{\mathbf 1}_{n})(P^{\tau})_{\tau^{-1}(\sigma)}= \sum_{\sigma \in T}\mu(\tau^{-1}(\sigma),\hat{\mathbf 1}_{n})(P_{\sigma})^{\tau}\\
			&= \sum_{\sigma \in T}\mu(\sigma,\hat{\mathbf 1}_{n})(P_{\sigma})^{\tau}= (T[P])^{\tau},
		\end{aligned}
		\]
		where the second equality occurs  since    $\pi \in T$ if and only if   $\tau^{-1}(\pi) \in T$, the third because  $(P_{\sigma})^{\tau}=(P^{\tau})_{\tau^{-1}(\sigma)}$ for any partition $\sigma$ and bijective $\tau$ and the fourth because   $\mu(\sigma,\hat{\mathbf 1}_{n}) = \mu(\tau^{-1}(\sigma),\hat{\mathbf 1}_{n})$.\\
		Conversely, suppose that $(T[P^{\tau}])=(T[P])^{\tau}$ for any bijection $\tau$ and probability $P$ of $\mathds{X}_{n}$. Since 
		\[
		(T[P^{\tau}])= \sum_{\pi \in T}\mu(\pi,\hat{\mathbf 1}_{n})(P^{\tau})_{\pi}, \quad  (T[P])^{\tau} =\sum_{\pi \in T}\mu(\pi,\hat{\mathbf 1}_{n})(P_{\pi})^{\tau} = \sum_{\pi \in T}\mu(\pi,\hat{\mathbf 1}_{n})(P^{\tau})_{\tau^{-1}(\pi)},
		\]
		by choosing a  marginally identifiable $P$ we obtain that $\mu(\pi,\hat{\mathbf 1}_{n})= \mu(\tau(\pi),\hat{\mathbf 1}_{n})$ for  all partitions $\pi$ and bijective $\tau$. To conclude,  the   maximal elements of  $ \{\pi \neq \hat{\mathbf 1}_{n}, \quad \mu(\pi,\hat{\mathbf 1}_{n}) \neq 0\}$  	are precisely the maximal elements of $T\setminus\{\hat{\mathbf 1}_{n}\}$, because  the  M\"obius coefficient of each such maximal element is $-1$. Hence  $\max(T)$, and therefore $T$, is invariant under 	the symmetric group.\end{proof}

	\section{Generalized Lancaster interaction}\label{GeneralizedLancasterinteraction}
	Using Equation \ref{latticegenerated}, define for $2\leq k \leq n $ the following symmetric partition lattice
	\[
	\Lambda_{k}^{n}:= \{ \pi \in \Pi_{n}, \quad \pi=\hat{\mathbf{1}}_{n} \text{ or }  \shp(\pi) \preceq (k-1,1^{n-k+1})  \},
	\]
	whose set of maximal elements is the set of $ \sigma  \in \Pi_{n}$ for which $\shp(\sigma)=(k-1,1^{n-k+1}) $.
	\begin{lemma} A partition $\pi$ belongs to $\Lambda_{k}^{n} \setminus\{\hat{\mathbf{1}}_{n}\}$ if and only if either  $\shp(\pi)= (1^{n})$ or  $\shp(\pi)= \alpha = (\alpha_{1}, \ldots, \alpha_{p}, 1^{n-a} )$, where $\alpha_{p} \geq 2$ and $a = \sum_{i=1}^{p}\alpha_{i} \leq  k-1$.
	\end{lemma}
	
	\begin{proof}Note that $(\alpha_{1}, \ldots, \alpha_{p}, 1^{n-a} ) \preceq (k-1,1^{n-k+1}) $ where $\alpha_{p} \geq 2$ and $a := \sum_{i=1}^{p}\alpha_{i}$  if and only if  $a \leq k-1$, because all terms $\alpha_{j}$ must be in the same block of the grouping, by the definition of the order on $\operatorname{Par}(n)$ in Definition  \ref{order[n]}.
	\end{proof}
	
	In order to obtain the  M\"obius function of $	\Lambda_{k}^{n}$, we prove some binomial relations.
	
	\begin{lemma}\label{binomlanck} The following identities hold when $N-1 \geq L$ and $n\geq k \geq a+1   $
		\begin{equation}\label{binomlanckeq1}
			\sum_{J=0}^{L}(-1)^{J}\binom{N-1-J}{L-J}\binom{N}{J}= (-1)^{L},
		\end{equation}
		\begin{equation}\label{binomlanckeq2}
			\sum_{j = a} ^{k-1}(-1)^{k-j}\binom{n-j-1}{n-k}\binom{n-a}{j-a}=-1,
		\end{equation}
		\begin{equation}\label{binomlanckeq3}
			\sum_{j=2}^{k-1}(-1)^{k-j}\binom{n-j-1}{n-k}\binom{n}{j} =-1 - (-1)^{k}\left [ \binom{n-1}{n-k} -  n\binom{n-2}{n-k} \right ].
		\end{equation}
	\end{lemma}
	\begin{proof}
		Note that
		\[
		\begin{aligned}
			(1+x)^{-1} &=  (1+x)^{N-1} \left ( 1 - \frac{x}{1+x}\right )^{N}= \sum_{J=0}^{N}(-1)^{J} \binom{N}{J}x^{J}(1+x)^{N-1-J}\\
			&= (-1)^{N}x^{N}(1+x)^{-1} +   \sum_{J=0}^{N-1}(-1)^{J} \binom{N}{J}x^{J}(1+x)^{N-1-J}.
		\end{aligned}
		\]
		The first term can be written as  $(-1)^{N}x^{N}(1+x)^{-1}= \sum_{L=N}^{\infty}(-1)^{L}x^{L}$, while the second term satisfies
		\[
		\begin{aligned}
			\sum_{J=0}^{N-1}(-1)^{J} \binom{N}{J}x^{J}(1+x)^{N-1-J}&= \sum_{J=0}^{N-1} \sum_{l=0}^{N-1-J}(-1)^{J} \binom{N-1-J}{l}\binom{N}{J}x^{J+l} \\
			&=\sum_{L=0}^{N-1} \left [ \sum_{J=0}^{L}(-1)^{J}\binom{N-1-J}{L-J}\binom{N}{J} \right ] x^{L}.
		\end{aligned}
		\]
		Thus, comparing coefficients of $x^{L}$ for $0\leq L\leq N-1$ with those in $(1+x)^{-1}=\sum_{L=0}^{\infty}(-1)^{L}x^{L}$, we obtain the desired equality.\\
		For the second relation, by Equation \ref{binomlanckeq1}   and the change of variables $J: =j-a$, $N : =n-a$ and $L:=k-a-1$, we obtain
		\[
		\begin{aligned}
			\sum_{j = a}^{k-1}(-1)^{k-j}\binom{n-j-1}{n-k}\binom{n-a}{j-a}&= \sum_{J = 0}^{L}(-1)^{L-J+1}\binom{N-J-1}{N-L-1}\binom{N}{J}\\
			&=	(-1)^{L+1}\sum_{J = 0}^{L}(-1)^{J}\binom{N-J-1}{L-J}\binom{N}{J}= -1.
		\end{aligned}
		\]
		To obtain Equation \ref{binomlanckeq3}, if we set $a=0$ in Equation \ref{binomlanckeq2}, we obtain
		\[
		-1 = (-1)^{k}\binom{n-1}{n-k} + (-1)^{k-1}n\binom{n-2}{n-k}	+\sum_{j = 2} ^{k-1}(-1)^{k-j}\binom{n-j-1}{n-k}\binom{n}{j},
		\]
		which concludes the proof.\end{proof}
	
	\begin{theorem}The M\"obius function at $ \pi=(F_{1}, \ldots, F_{\ell}) \in \Lambda_{k}^{n} $ is
		\[
		\mu(\pi, \hat{\mathbf{1}}_{n})=
		\begin{cases}
			1, &  \text{ if } \pi = \hat{\mathbf{1}}_{n}\\
			0,& \text{ if } |F_{2}| \geq 2,\\
			(-1)^{k-|F_{1}|}\binom{n-|F_{1}|-1}{n-k}, & \text{ if } |F_{2}|=1 \text{ and } k-1 \geq |F_{1}|\geq 2,\\
			(-1)^{k}\left[\binom{n-1}{n-k} - n\binom{n-2}{n-k}  \right ]& \text{ if } \pi= \hat{\mathbf{0}}_{n}
		\end{cases}
		\]
	\end{theorem}
	
	\begin{proof} We prove the formula by induction on the level sets of the set $S_{k} \subset \operatorname{Par}(n)$ obtained through Theorem \ref{charsymmetricpartitionposet}.\\ By definition  $H((n))=\mu(\hat{\mathbf{1}}_{n}, \hat{\mathbf{1}}_{n})=1$. Now, if $\alpha \in I_{2}$ then $\alpha= (k-1, 1^{n-k+1})$, hence,  $H(\alpha) =-1$.\\
		Now,  let  $\alpha=(\alpha_{1}, \ldots, \alpha_{p}, 1^{n-a} ) \in I_{h}$  with $\alpha_{p}\geq2$ and $a=\sum_{i=1}^{p}\alpha_{i}$ (more precisely, although this is not needed for the proof, $h=p+k-a $).
		By the induction hypothesis, if  $\beta= (\beta_{1}, \ldots, \beta_{q}, 1^{n-b}) \in 	S_{k}\setminus{\{(n)\}} $ with $\beta_{q}\geq 2$ and $\beta \succ \alpha$, then either $q\geq 2$ and  $H(\beta)=0$, or $q=1$ and the possible shapes are   $\beta= (j, 1^{n-j})$ for some $j=a+1, \ldots, k-1$ (if $p=1$)  or $\beta= (j, 1^{n-j})$ for some $j=a, \ldots, k-1$ (if $p\geq 2$). Hence, by  using Equation \ref{binomlanckeq2} for the case $p=1$
		\[
		\begin{aligned}
			H(\alpha)&= -\sum_{\beta \succ \alpha}H(\beta)m(\alpha, \beta)=-H((n))m(\alpha, (n)) -\sum_{j=a+1}^{k-1}H((j, 1^{n-j}))m(\alpha, (j, 1^{n-j}))\\
			&=-1   -\sum_{j=a+1}^{k-1}(-1)^{k-j}\binom{n-j-1}{n-k}\binom{n-a}{n-j}=(-1)^{k-a}\binom{n-a-1}{n-k}.
		\end{aligned}
		\]
		The case $p\geq 2$ is also a consequence of Equation \ref{binomlanckeq2}, as
		\[
		\begin{aligned}
			H(\alpha)&= -\sum_{\beta \succ \alpha}H(\beta)m(\alpha, \beta)=-H((n))m(\alpha, (n)) -\sum_{j=a}^{k-1}H((j, 1^{n-j}))m(\alpha, (j, 1^{n-j}))\\
			&=-1   -\sum_{j=a}^{k-1}(-1)^{k-j}\binom{n-j-1}{n-k}\binom{n-a}{n-j}=0.
		\end{aligned}
		\]
		Lastly, if $\alpha=(1^{n})$,  we may only sum over the terms $\beta=(j, 1^{n-j})$ for  $2 \leq j \leq k-1$, by   Equation  \ref{binomlanckeq3}
		\[
		\begin{aligned}
			H((1^{n}))&=- H((n))m((1^{n}), (n)) - \sum_{j=2}^{k-1}H((j, 1^{n-j})) m((1^{n}),(j, 1^{n-j}))\\
			&= -1 -  \sum_{j=2}^{k-1}(-1)^{k-j}\binom{n-j-1}{n-k}\binom{n}{j} =  (-1)^{k}\left [ \binom{n-1}{n-k} -  n\binom{n-2}{n-k} \right ].
		\end{aligned}
		\]
	\end{proof}
	
	Hence, as defined in Equation \ref{interaction}, for a probability  $P \in \mathcal{M}(\mathds{X}_{n})$, we define its interaction with respect to $\Lambda_{k}^{n}$ as
	\[
	\Lambda_{k}^{n}[P] = \sum_{\pi \in \Lambda_{k}^{n}}\mu(\pi,\hat{\mathbf1}_{n})P_{\pi}= P + \sum_{\ell=1}^{k-1}H((\ell,1^{n-\ell})) \sum_{ \shp(\pi)= (\ell, 1^{n-\ell}) }P_{\pi}.
	\]
	However, for $\ell \geq 2$
	\[
	H((\ell,1^{n-\ell}))\sum_{ \shp(\pi)= (\ell, 1^{n-\ell}) }P_{\pi}= (-1)^{k-\ell}\binom{n-\ell-1}{n-k}\sum_{|F|=\ell}P_{F}\times \left [\bigtimes_{i \in F^{c}}P_{i} \right ]
	\]
	and the case $\ell=1$ can be conveniently rewritten as
	\[
	H((1^{n}))\left [\bigtimes_{i=1}^{n}P_{i}\right ]= \sum_{\ell=0}^{1}(-1)^{k-\ell}\binom{n-\ell-1}{n-k}\sum_{|F|=\ell}P_{F}\times \left [\bigtimes_{i \in F^{c}}P_{i} \right ].
	\]
	Thus, we obtain the final and simplified expression
	\[
	\Lambda_{k}^{n}[P] = P + \sum_{\ell=0}^{k-1}(-1)^{k-\ell}\binom{n-\ell-1}{n-k}\sum_{|F|=\ell}P_{F}\times \left [\bigtimes_{i \in F^{c}}P_{i} \right ],
	\]
	which was initially proposed in   \cite{Guella2025a}. The same reference also presents a generalization of this definition for which we do not currently have a lattice interpretation. For probability measures $P,Q \in \mathcal{M}(\mathds{X}_{n})$, it is defined by
	\[
	\Lambda_{k}^{n}[P,Q] := P + \sum_{\ell=0}^{k-1}(-1)^{k-\ell}\binom{n-\ell-1}{n-k}\sum_{|F|=\ell}P_{F}\times Q_{F^{c}},
	\]
	and $	\Lambda_{k}^{n}[P]=\Lambda_{k}^{n}[P, \bigtimes_{i =1}^{n}P_{i}] $.
	
	The following properties were proved in \cite{Guella2025a} without using the lattice interpretation developed here. In property $(iii)$, we use that $\Lambda_{1}^{1}[Q]=0$ for a one-coordinate probability measure $Q$.
	
	\begin{theorem}\label{generallancaster} The generalized Lancaster interaction $\Lambda_{k}^{n}$ satisfies the following properties:
		\begin{enumerate}
			\item [$i)$] For discrete probability measures $P, Q \in \mathcal{M}(\mathds{X}_{n})$, we have $\Lambda_{k}^{n}[P,Q]\in\mathcal M_{k}(\mathds X_{n})$.
			\item [$ii)$] If for some $2\leq k \leq n-1$ we have $\Lambda_{k}^{n}[P]=0$, then $\Lambda_{k+1}^{n}[P]=0$.
			\item [$iii)$] $\Lambda_{n}^{n}[P]$ is multiplicative in the sense that if $P= P_{\pi}$ for some partition $\pi= (F_{1}, \ldots, F_{\ell})$ of $\{1, \ldots, n\}$, then $\Lambda_{n}^{n}[P]= \bigtimes_{i=1}^{\ell}\Lambda_{|F_{i}|}^{|F_{i}|}[P_{F_{i}}]$.
			\item[$iv)$] For any $G \subset\{1, \ldots, n\}$ with $|G|\geq k$, we have $\left ( \Lambda_{k}^{n}[P]\right )_{G}= \Lambda_{k}^{|G|}[P_{G}]$.
			\item [$v)$]  The symmetric partition lattice $\Lambda_{k}^{n}$ satisfies Theorem \ref{genStreitbergvectorspace} and Theorem \ref{genStreitbergcartesianproduct}.
		\end{enumerate}
	\end{theorem}
	
	We emphasize relation $ii)$, for which we have not yet found another family of interactions sharing a similar property. We also emphasize the decomposition in property $(iii)$.
	
	The radial kernel characterizations and their simplifications in \cite{Guella2025a} give the following equivalences.
	
	\begin{theorem}\label{bernsksevndimpart3} Let $n\geq k\geq2$, and let $g:[0, \infty)^{n} \to \mathbb{R}$ be a continuous function. The following conditions are equivalent:
		\begin{enumerate}
			\item [$(i)$] For any $d\in \mathbb{N}$ and discrete signed  measures $\mu_{i}$ on $\mathbb{R}^{d}$, $1\leq i \leq n$, with the restriction that $|\{i, \quad \mu_{i}(\mathbb{R}^{d})=0\}| \geq k $,  it holds that
			\[
			\int_{(\mathbb R^{d})^{n}}\int_{ (\mathbb R^{d})^{n}}(-1)^{k}g(\|x_{1}-y_{1}\|^{2}, \ldots, \|x_{n} - y_{n}\|^{2})d[\bigtimes_{i=1}^{n}\mu_{i}](x)d[ \bigtimes_{i=1}^{n}\mu_{i}](y)\geq 0.
			\]
			\item [$(ii)$] For any $d\in \mathbb{N}$ and nonzero discrete probability measure $P$ on $(\mathbb R^{d})^{n}$,  it holds that
			\[
			\int_{(\mathbb R^{d})^{n}}\int_{ (\mathbb R^{d})^{n}}(-1)^{k}g(\|x_{1}-y_{1}\|^{2}, \ldots, \|x_{n} - y_{n}\|^{2})d[\Lambda_{k}^{n}[P] ](x)d[\Lambda_{k}^{n}[P]](y)\geq 0.
			\]
			\item [$(iii)$]  For any $d\in \mathbb{N}$ and discrete signed measure $\eta$ in $\mathcal M_{k}((\mathbb R^{d})^{n})$,  it holds that
			\[
			\int_{(\mathbb R^{d})^{n}}\int_{ (\mathbb R^{d})^{n}}(-1)^{k}g(\|x_{1}-y_{1}\|^{2}, \ldots, \|x_{n} - y_{n}\|^{2})d\eta(x)d\eta (y)\geq 0.
			\]
		\end{enumerate}
	\end{theorem}
	
	By Theorems \ref{genStreitbergcartesianproduct} and \ref{genStreitbergvectorspace}, we may replace $\Lambda_{k}^{n}$ in condition $(ii)$ by any symmetric partition lattice $T$ of order $k$. Indeed, condition $(iii)$ implies the resulting inequality because $T[P]\in\mathcal M_{k}$, and that inequality implies condition $(i)$ by the representation in Theorem \ref{genStreitbergcartesianproduct}. Thus lattices of the same order define the same nonnegativity conditions in this setting.  We expect analogous equivalences for isotropic kernels on spheres and compact homogeneous spaces.
	
	\section{Generalized Streitberg interaction}\label{GeneralizedStreitberginteraction}
	
	In this section we define a new symmetric partition lattice based on the Streitberg lattice and obtain its M\"obius function. For $2 \leq k \leq n$ we define the lattice
	\[
	\Sigma_{k}^{n}:=  \{ \pi \in \Pi_{n}, \quad |\pi| \geq n-k+2 \text{ or } \pi=\hat{\mathbf{1}}_{n}\},
	\]
	that is,  the set consisting of the largest partition and all partitions having at least $n-k+2$ blocks. This set satisfies Theorem \ref{charsymmetricpartitionposet}. Note that when $k=n$ we have that $\Sigma_{n}^{n} = \Pi_{n}$.
	
	To obtain its M\"obius function, first, we recall the recurrence relations for the Stirling numbers of the first kind (the signed version)
	
	\begin{equation}\label{Stirling1}
		s(n+1,k+1)= -ns(n, k+1) + s(n, k), \quad n, k \geq 0
	\end{equation}
	with the initial conditions
	\[
	s(0,0)=1 \text{ and } s(n+1, 0 )=s(0,n+1 )=0 \text{ for all } n \geq 0.
	\]
	
	We also recall the recurrence relation for the  Stirling numbers of the  second kind
	\begin{equation}\label{Stirling2}
		S(n+1,k+1)= (k+1)S(n, k+1) + S(n, k), \quad n, k \geq 0
	\end{equation}
	with the same initial conditions
	\[
	S(0,0)=1 \text{ and } S(n+1, 0 )=S(0,n+1 )=0 \text{ for all } n \geq 0.
	\]
	
	In particular,
	\begin{equation}\label{lower}
		s(n,k)=S(n,k)=0, \quad  \text{ if }k>n.
	\end{equation}
	
	These two objects are interconnected, as the matrices of Stirling numbers of the first and second kind are inverses of each other, in the sense that
	
	\begin{equation}\label{stirling1and2inverse}
		\sum_{j=0}^{n}s\left(j, \ell \right)S\left(L,j\right)=\sum_{j=0}^{n}s\left(L,j\right)S\left(j,\ell \right)=\delta_{L,\ell}, \quad n\geq \ell, L.
	\end{equation}
	
	We also   recall the following expansion of the signed Stirling number of the first kind using the falling factorial
	\[
	(x)_{n}:=\prod_{k=0}^{n-1}(x-k)= \sum_{m=0}^{n}s(n,m)x^{m},
	\]
	Setting $x=1$ for $n\geq2$ gives the following relation, which we use in the text
	\begin{equation}\label{fallingfactorial}
		\sum_{m=0}^{n}s(n,m)=\sum_{m=1}^{n}s(n,m)=0, \quad n>1.
	\end{equation}
	
	In order to prove our main results for $\Sigma_{k}^{n}$, we need a few additional  equalities regarding the connection between Stirling numbers.
	
	\begin{lemma}\label{prestreit} For $\ell, j \geq 0$ we define $a(j,\ell):= \sum_{m=1}^{\ell+1}s(j,m)$.
		\begin{enumerate}
			\item[$i)$] It satisfies the same recurrence as in Equation \ref{Stirling1}
			\[
			a(j+1,\ell+1 )= a(j,\ell) -	ja(j,\ell+1) , \quad  j\geq1,\ \ell\geq0
			\]
			with initial conditions  $a(1,\ell)=1$, $a(j,0)=s(j,1)=(-1)^{j-1}(j-1)!$ for $j,\ell \geq 1$.
			\item [$ii)$] For $1\leq L\leq n$ and $\ell\geq0$,
			\[
			\sum_{j=1}^{n}a(j,\ell)S(L,j) = \begin{cases}
				1 & \text{if } 1 \leq L \leq \ell +1 \\
				0 & \text{if } L \geq \ell +2
			\end{cases}
			\]
			\item[$iii)$] $a(j,\ell)=0$ for $1 < j \leq \ell+1$  and $a(\ell +2,\ell)=-1$.
		\end{enumerate}
	\end{lemma}
	
	\begin{proof}
		To prove relation $i)$, note that by Equation \ref{Stirling1}
		\[
		\begin{aligned}
			ja(j,\ell+1) + a(j+1,\ell+1 )&= \sum_{m=1}^{\ell +2}[js(j,m) + s(j+1,m)]\\
			&= \sum_{m=1}^{\ell +2}[js(j,m) + [-js(j,m) + s(j,m-1)]]\\
			&=	\sum_{m=1}^{\ell +2}  s(j,m-1)= 0 + \sum_{m=1}^{\ell +1}  s(j,m)= a(j,\ell).
		\end{aligned}
		\]
		The boundary cases are obtained by direct verification. \\
		For relation $ii)$, by Equation \ref{stirling1and2inverse}
		\[
		\sum_{j=1}^{n}a(j,\ell)S(L,j)=\sum_{j=1}^{n}\left [\sum_{m=1}^{\ell +1}s(j,m)\right ]S(L,j)=\sum_{m=1}^{\ell +1}\left [	\sum_{j=1}^{n}s(j,m) S(L,j)\right ]=\sum_{m=1}^{\ell +1}\delta_{m,L},
		\]
		and the conclusion follows directly from this.\\
		To prove relation $iii)$,  by Equation \ref{fallingfactorial} and Equation \ref{lower} we get
		\[
		a(j, \ell)=\sum_{m=1}^{\ell+1}s(j,m)= \sum_{m=1}^{j}s(j,m)+ \sum_{m=j+1}^{\ell+1}s(j,m)=0+0=0, \quad 1< j \leq \ell +1
		\]
		\[
		a(\ell+2,\ell)=\sum_{m=1}^{\ell+1}s(\ell+2,m)= \sum_{m=1}^{\ell+2}s(\ell+2,m) -  s(\ell+2,\ell+2) =0-1=-1.
		\]
	\end{proof}
	
	For a probability measure $P$ on $\mathds{X}_{n}$, define for $1\leq k \leq n$  the nonnegative measure $
	P^{k,n}:= \sum_{|\pi|=k}P_{\pi}$, that is, $P^{k,n}$ is the sum over all partitions of $\{1, \ldots, n\}$  having $k$ blocks. For convenience, define $P^{0,n}$ and $P^{n+h,n}$ as the zero measure for every $n,h\geq 1$. Note that $P^{k,n}(\mathds{X}_{n})= S(n,k)$.
	
	\begin{theorem}  \label{genstreitchar} The M\"obius function at $ \pi  \in 	\Sigma_{k}^{n} $ is $ \mu(\pi, \hat{\mathbf{1}}_{n})=a(|\pi|, n-k)$.
	\end{theorem}
	
	\begin{proof}
		If $|\pi|=1$, then $a(1,n-k)=1=\mu(\pi,\hat{\mathbf1}_{n})$. If $|\pi|=n-k+2$, then $\pi$ is maximal below the top, so both coefficients are $-1$ by Lemma \ref{prestreit}. For the remaining cases, we verify the proposed coefficients in the M\"obius recursion.
		Let $|\pi|=\ell>n-k+2$. Substituting the proposed value $a(|\eta|,n-k)$ into the sum over $\eta\geq\pi$ gives
		\[
		1+\sum_{\substack{\hat{\mathbf1}_{n}>  \eta\geq\pi\\\eta\in\Sigma_{k}^{n}}}a(|\eta|,n-k)=1+\sum_{j=n-k+2}^{\ell}a(j,n-k)S(\ell,j)=\sum_{j=1}^{\ell}a(j,n-k)S(\ell,j)=0.
		\]
		The first equality counts the coarsenings with $j$ blocks. The second uses $a(1,n-k)=1$ and $a(j,n-k)=0$ for $2\leq j\leq n-k+1$; the last follows from Lemma \ref{prestreit}$(ii)$. Thus the proposed coefficients satisfy Equation \ref{somax1}, which uniquely determines the M\"obius coefficients at the top.
	\end{proof}
	
	It is also possible to prove Theorem \ref{genstreitchar}  from Theorem \ref{charsymmetricpartitionposet} and the equality $ \sum_{\ell(\beta)=j}m(\alpha, \beta)=S(\ell(\alpha),j)$.
	Hence, as defined in Equation \ref{interaction}, for a probability  $P \in \mathcal{M}(\mathds{X}_{n})$, we define its interaction with respect to $\Sigma_{k}^{n}$ as
	\begin{equation}\label{streitkdefn}
		\Sigma_{k}^{n}[P]:=\sum_{\pi \in \Sigma_{k}^{n}}\mu(\pi, \hat{\mathbf{1}}_{n})P_{\pi}=    \sum_{j=1}^{n}a(j, n-k)P^{j,n},
	\end{equation}
	where the last equality is obtained using property    $iii)$ in Lemma \ref{prestreit}.
	
	\begin{corollary}\label{consequencesstreik} The generalized Streitberg interaction $\Sigma_{k}^{n}$ satisfies the following properties:
		\begin{enumerate}
			\item[$i)$] The symmetric partition lattice $\Sigma_{k}^{n}$ has order $k$.
			\item[$ii)$]  If $F \subset \{1, \ldots, n\}$ with $|F| \geq k$ then $(\Sigma_{k}^{n}[P])_{F} = \Sigma_{k}^{|F|}[P_{F}] $.
			\item[$iii)$] $\Sigma_{n}^{n}[P]$ is  the standard definition of the Streitberg interaction.
		\end{enumerate}
	\end{corollary}
	
	\begin{proof}
		If $\alpha=(\alpha_{1}, \ldots, \alpha_{\ell}) \in \operatorname{Par}(n)$ satisfies $\alpha=\shp(\pi)$ for some $\pi \in \Sigma_{k}^{n}\setminus\{\hat{\mathbf{1}}_{n}\}$, then we must have by the definition of $\Sigma_{k}^{n}$ that $\ell \geq n-k+2$, and, since all remaining parts are at least $1$, the largest possible value of $\alpha_{1}$ must be $k-1$, which occurs when $\alpha=(k-1, 1^{n-k+1})$, thus concluding relation $i)$ as $|\Sigma_{k}^{n}|_{\infty}= k-1 +1=k$.\\
		For property $ii)$, the case $k=n$ is immediate. Suppose $k\leq n-1$. Given a partition $\pi$ of $\{1, \ldots, n\}$ of size $j$, the partition $\pi^{\prime}$ of $\{1, \ldots, n-1\}$ obtained by removing the element $n$ from $\pi$ either has size $j$ (in this case $n$ was inserted into one of the $j$ blocks of $\pi^{\prime}$) or has size $j-1$ (in this case the element $n$ is a singleton of $\pi$). From this analysis, we obtain that
		\[
		(P^{j,n})_{\{1, \ldots, n-1\}}= j(P_{\{1, \ldots, n-1\}})^{j,n-1 } + (P_{\{1, \ldots, n-1\}})^{j-1,n-1 }, \quad  n,j \geq 1.
		\]
		By relabeling the coordinates, we obtain that the previous equation also holds for any subset $F\subset \{1, \ldots, n\}$ with $|F|=n-1$. Hence, for such $F$, relation $i)$ in Lemma \ref{prestreit} gives
		\[
		\begin{aligned}
			(\Sigma_{k}^{n}[P])_{F}&= \sum_{j=1}^{n}a(j, n-k)(P^{j,n})_{F}=  \sum_{j=1}^{n}a(j, n-k)[j(P_{F})^{j,n-1 } + (P_{F})^{j-1,n-1}]\\
			&= \sum_{l=1}^{n-1}(a(l+1,n-k) + la(l,n-k)  ) (P_{F})^{l,n-1 } = \sum_{l=1}^{n-1}a(l,n-k-1)(P_{F})^{l,n-1 }= \Sigma_{k}^{n-1}[P_{F}].
		\end{aligned}
		\]
		Applying this result recursively, we obtain the general case.\\
		For relation $iii)$, since $a(j,n-n)=a(j,0)= s(j, 1)=(-1)^{j-1}(j-1)! $, we obtain that $\Sigma_{n}^{n}[P]=  \sum_{j=1}^{n}(-1)^{j-1}(j-1)!P^{j,n}$.\end{proof}

	\section{Size-limited symmetric lattices}\label{sizelimited}
	
	We define for $2 \leq k\leq n$
	\[
	J_{k }^{n}:=\{\pi \in \Pi_{n}, \quad \shp(\pi)_{1} \leq k-1 \} \cup \{\hat{\mathbf{1}}_{n}\}.
	\]
	This  is a symmetric partition lattice, because if $\pi, \sigma \in  \Pi_{n}$ with $\pi \leq \sigma$, then $ \shp(\pi)_{1}\leq   \shp(\sigma)_{1}$. We focus our analysis on the case $k=3$, for which the maximal elements have the following shapes
	\[
	(2^{N},1)\quad\text{when }n=2N+1,\qquad (2^{N})\quad\text{when }n=2N.
	\]
	For $\lceil n/2 \rceil \leq S \leq n-1$, we further define
	\[
	J_{3, S }^{n}:=\{\pi \in \Pi_{n}, \quad \shp(\pi)_{1} \leq 2 \text{ and } |\pi|\geq S\} \cup \{\hat{\mathbf{1}}_{n}\},
	\]
	which can be rewritten as
	\[
	J_{3, S }^{n}:=\{\pi \in \Pi_{n}, \quad \shp(\pi) =(2^{n-S -j}, 1^{2S -n+2j} ) \text { for } 0 \leq j \leq  n-S \} \cup \{\hat{\mathbf{1}}_{n}\},
	\]
	whose maximal elements are 	$\{\pi\in\Pi_{n}:\shp(\pi)=(2^{n-S},1^{2S-n})\}$.  For instance, $J_{3,3}^{5}$ contains partitions of the following non-top shapes $(2,2,1)$, $(2,1,1,1)$ and $(1,1,1,1,1)$ while  $J_{3,4}^{5}$ contains partitions of the following non-top shapes  $(2,1,1,1)$ and $(1,1,1,1,1)$.
	
	Below, we compute some  terms $m(\alpha, \beta)$ that will be needed when using Theorem \ref{charsymmetricpartitionposet}.
	
	\begin{lemma}\label{computem} For $ N \geq 1$  and all $0\leq l \leq j\leq N$,   it holds that
		
		\[
		m(  (2^{N-j}, 1^{2j+1} ), (2^{N-l}, 1^{2l+1} ))= \binom{2j+1}{2l+1}\frac{(2j -2l)!}{2^{j-l}(j-l)!}
		\]
		
		\[
		m( (2^{N-j}, 1^{2j} ), (2^{N-l}, 1^{2l} ))=   \binom{2j}{2l}\frac{(2j -2l)!}{2^{j-l}(j-l)!}
		\]
		
	\end{lemma}
	
	\begin{proof} In the first equality,  to go from the  element $(2^{N-j}, 1^{2j+1})$ to  $(2^{N-l}, 1^{2l+1} )$ we must choose which of the $2j+1$ entries equal to $1$ will be paired to form the remaining $j-l$ parts of size $2$  (or equivalently,  which $2l+1$ of the $2j+1$ entries will remain equal to $1$) which is exactly $\binom{2j+1}{2l+1}$. Now that we have the $2(j-l)$ elements to form pairs, we must count the ways in which they can be paired, which is precisely $\frac{(2j -2l)!}{2^{j-l}(j-l)!}$, which concludes the first equality. The proof of the second equality is similar and is omitted.
	\end{proof}
	
	A small and useful reformulation of the identities of Lemma \ref{computem} is that for $\lceil n/2 \rceil  \leq q \leq p \leq n$
	
	\begin{equation}\label{computemuseful}
		m( (2^{n-p}, 1^{2p-n} ), (2^{n-q}, 1^{2q-n} ))=   \binom{2p-n}{2q-n}\frac{(2p -2q)!}{2^{p-q}(p-q)!}.
	\end{equation}
	
	For $a\in\mathbb Z_{\geq0}$, define the sequence recursively by $b_{0}^{a}:=-1$, $b_{1}^{a}:=-1$, and  for $j\geq 1$
	\begin{equation}\label{oddtermsgeneral}
		b_{2j+1}^{a}:= -1 -\sum_{l=0}^{j-1}b_{2l+1}^{a}\binom{a+2j+1}{a + 2l+1}\frac{(2j -2l)!}{2^{j-l}(j-l)!},
	\end{equation}
	\begin{equation}\label{eventermsgeneral}
		b_{2j}^{a}:= -1 -\sum_{l=0}^{j-1}b_{2l}^{a}\binom{a+2j}{a + 2l}\frac{(2j -2l)!}{2^{j-l}(j-l)!}.
	\end{equation}
	
	\begin{theorem}\label{recursivej3lmain}
		The sequence $(	b_{j}^{a})_{j\geq0}$  defined in Equations \ref{oddtermsgeneral} and  \ref{eventermsgeneral} satisfies
		\begin{equation}\label{recursivej3lmaineq}
			-e^{-x^{2}/2}\left [ \frac{e^{x} - \sum_{c=0}^{a-1}x^{c}/ c!}{x^{a}}  \right ]= \sum_{j=0}^{\infty}\frac{b^{a}_{j}}{(a+j)!}x^{j}, \quad x \in \mathbb{R}.
		\end{equation}
		Here the quotient is extended continuously at $x=0$, with value $1/a!$, and the sum is empty when $a=0$.
	\end{theorem}
	
	\begin{proof}	First, note that from Equation  \ref{oddtermsgeneral}, we have that for any $j \geq 1$
		\[		1=- \sum_{l=0}^{j}b^{a}_{2l +1}\binom{a+2j+1}{a+2l+1}\frac{(2j -2l)!}{2^{j-l}(j-l)!}.
		\]
		Dividing both sides by $(a+2j+1)!$ and simplifying the coefficients, we obtain
		\begin{equation}\label{oddtermsrecurrence2general}
			\frac{1}{(a+2j+1)!}= -\sum_{l=0}^{j}\frac{b^{a}_{2l+1}}{(a+2l+1)!}\frac{1}{2^{j-l}(j-l)!}.
		\end{equation}
		Similarly, using Equation \ref{eventermsgeneral} we also obtain
		\begin{equation}\label{eventermsrecurrence2general}
			\frac{1}{(a+2j)!}= -\sum_{l=0}^{j}\frac{b^{a}_{2l}}{(a+2l)!}\frac{1}{2^{j-l}(j-l)!}.
		\end{equation}
		Define the following formal power series and recall the other two expansions
		\[
		B^{a}(x):=\sum_{j=0}^{\infty}\frac{b^{a}_{j}}{(a+j)!}x^{j}, \quad e^{x^{2}/2}=\sum_{j=0}^{\infty} \frac{1}{2^{j}j!}x^{2j}, \quad E_{a}(x):= \frac{e^{x} - \sum_{c=0}^{a-1}x^{c}/ c!}{x^{a}}= \sum_{j=0}^{\infty} \frac{1}{(a+j)!}x^{j}.
		\]
		By Equations \ref{oddtermsrecurrence2general} and \ref{eventermsrecurrence2general}, which also hold for $j=0$ by the initial conditions, the coefficient of $x^{m}$ in the Cauchy product $B^{a}(x)e^{x^{2}/2}$ equals $-1/(a+m)!$. Therefore
		\[
		B^{a}(x)e^{x^{2}/2}=-E_{a}(x),\qquad B^{a}(x)=-E_{a}(x)e^{-x^{2}/2}.
		\]
		The function $E_{a}$ has a removable singularity at zero and is entire. Hence the last expression is entire, so its Taylor series converges for every real $x$ and proves Equation \ref{recursivej3lmaineq}.
	\end{proof}
	
	\begin{theorem} For $0\leq j \leq n-S$, the M\"obius function at $ \pi  \in 	J_{3,S}^{n} $ such that $ \shp(\pi) =(2^{n-S -j}, 1^{2S -n+2j} )  $ satisfies $\mu(\pi, \hat{\mathbf{1}}_{n})=   b^{2S-n}_{2j}$.
	\end{theorem}
	\begin{proof} By definition, $H((n))=1$. Since $(2^{n-S}, 1^{2S-n})$ is a maximal non-top shape, 	\\$H((2^{n-S}, 1^{2S-n}))=-1$. Since $b^{a}_{0}= -1$ for any $a \in \mathbb Z_{\geq0}$ by definition, we obtain the equality for $j=0$.\\
		By induction on $j$ and Equation \ref{computemuseful}, we conclude that
		\[
		\begin{aligned}
			&H((2^{n-S-j}, 1^{2S-n+2j}))\\
			& = - H((n))  -\sum_{l=0}^{j-1}H((2^{n-S-l}, 1^{2S-n+2l}))m((2^{n-S-j}, 1^{2S-n+2j}), (2^{n-S-l}, 1^{2S-n+2l}))\\
			&= -1 -\sum_{l=0}^{j-1} b^{2S-n}_{2l}\binom{2S-n+2j}{2S-n+2l}\frac{(2j-2l)!}{2^{j-l}(j-l)!} = b^{2S-n}_{2j}.
		\end{aligned}
		\]
	\end{proof}
	
	Hence, as defined in Equation \ref{interaction}, for a probability  $P \in \mathcal{M}(\mathds{X}_{n})$, we define its interaction with respect to $J_{3,S}^{n}$ as
	\begin{equation}\label{sizekdefn}
		J_{3,S}^{n}[P]:=\sum_{\pi \in J_{3,S}^{n}}\mu(\pi, \hat{\mathbf{1}}_{n})P_{\pi}=  P + \sum_{j=0}^{n-S}b^{2S-n}_{2j}P^{(2^{n-S -j}, 1^{2S -n+2j})}
	\end{equation}
	
	where $P^{(2^{J}, 1^{L})}:= \sum_{\shp(\pi)=(2^{J}, 1^{L}) }P_{\pi}$.
	
	In what follows, we prove recurrence relations for the terms $b_{j}^{a}$ that are useful for obtaining the marginals of the interaction defined by $J_{3,S}^{n}$.
	
	\begin{lemma}\label{recurrenceJ}
		The following identities hold:
		\[
		\begin{aligned}
			b_{j}^{a}+(a+j-1)b_{j-2}^{a}&=b_{j}^{a-1} &&(a\geq1,\ j\geq2),\\
			b_{2j}^{a}&=b_{2j+1}^{a-1} &&(a\geq1,\ j\geq0),\\
			b_{j}^{0}+(j-1)b_{j-2}^{0}&=b_{j-1}^{0} &&(j\geq2).
		\end{aligned}
		\]
	\end{lemma}
	
	\begin{proof}
		Taking the derivative of both sides of Equation \ref{recursivej3lmaineq} when $a\geq 1$,  we get that
		\[
		xe^{-x^{2}/2}E_{a}(x) - e^{-x^{2}/2}\left [\frac{E_{a-1}(x)}{x} - a\frac{E_{a }(x)}{x}  \right ]= \sum_{j=0}^{\infty}\frac{(j+1)b^{a}_{j+1}}{(a+j+1)!}x^{j}.
		\]
		Multiplying both sides by $x$ and using the  series expansion of Equation \ref{recursivej3lmaineq}, we obtain
		\[
		-\frac{b_{j-2}^{a}}{(a+j-2)!} + \frac{b_{j}^{a-1}}{(a+j-1)!} - a\frac{b_{j}^{a}}{(a+j)!}= j\frac{b_{j}^{a}}{(a+j)!},
		\]
		which implies the first equation. The equality $b_{2j}^{a}=b_{2j+1}^{a-1}$ holds when $j=0$, as both sides are equal to $-1$ by definition, and by induction  using Equations \ref{oddtermsgeneral} and \ref{eventermsgeneral} we get that
		\[
		b_{2j+1}^{a-1}= -1 -\sum_{l=0}^{j-1}b_{2l+1}^{a-1}\binom{a+2j}{a + 2l}\frac{(2j -2l)!}{2^{j-l}(j-l)!}= -1 -\sum_{l=0}^{j-1}b_{2l}^{a}\binom{a+2j}{a + 2l}\frac{(2j -2l)!}{2^{j-l}(j-l)!}= b_{2j}^{a}.
		\]
		The last equation follows by the same argument used for the first identity, so the proof is omitted.
	\end{proof}
	A small and useful reformulation of the identities of Lemma \ref{recurrenceJ} is that, for $2S>n$ and $j\geq1$,
	\begin{equation}\label{computemusefulbcoef}
		b_{2j}^{2S-n} + (2S-n +2(j-1) + 1)b_{2(j-1)}^{2S-n}= b_{2j}^{2S-n-1}= b_{2j}^{2(S-1)-(n-1)},
	\end{equation}
	For $2S=n$ and $j\geq1$, we instead have $b_{2j}^{0} + (2j-1)b^{0}_{2(j-1)}=b_{2j-1}^{0}=b^{1}_{2(j-1)}$.
	
	\begin{corollary} The following properties are satisfied:
		\begin{enumerate}
			\item[$i)$] The symmetric partition lattice $J_{k }^{n}$ has order $k$ and $J_{3, S}^{n}$ has order $3$.
			\item[$ii)$] For $n\geq4$ and $F\subseteq\{1,\ldots,n\}$ with $|F|=n-1$, we have $(J_{3,S}^{n}[P])_{F}=J_{3,S-1}^{n-1}[P_{F}]$ if $2S>n$, and $(J_{3,S}^{n}[P])_{F}=J_{3,S}^{n-1}[P_{F}]$ if $2S=n$. For $n=3$, the marginals of $J_{3,2}^{3}[P]$ are zero.
			
		\end{enumerate}
	\end{corollary}
	
	\begin{proof} If $2S >n$, by an analysis similar to that in relation $ii)$ of Corollary \ref{consequencesstreik}, the following identity holds for $2j+l =n$, where $l \geq 1$, and any $F\subset \{1, \ldots, n\}$ with $|F|=n-1$
		\[
		(P^{(2^{j}, 1^{l})})_{F}= (P_{F})^{(2^{j}, 1^{l-1})}+ (l+1)(P_{F})^{(2^{j-1}, 1^{l+1})}, \quad \text{ for } j\geq 1, \quad \text{ and } (P^{(1^{n})})_{F}= (P_{F})^{( 1^{n-1})}.
		\]
		For instance, if $F=\{1, \ldots, n-1\}$, for a related partition, either the element $n$ is a singleton (which gives a unique extension of each resulting partition), or it belongs to a pair (which does not occur if $j=0$), and in this case, the other element of the pair will become a singleton after taking the marginal, and each resulting partition of $\Pi_{n-1}$ occurs $l+1$ times, once for each singleton.  Since
		\[
		(J_{3,S}^{n}[P])_{F} = P_{F} + \sum_{j=0}^{n-S}b^{2S-n}_{2j}(P^{(2^{n-S -j}, 1^{2S -n+2j})})_{F},
		\]
		this marginal relation and Equation \ref{computemusefulbcoef} lead to
		\[
		\begin{aligned}
			& \sum_{j=0}^{n-S}b^{2S-n}_{2j}(P^{(2^{n-S -j}, 1^{2S -n+2j})})_{F}\\
			&=b_{0}^{2S-n} (P_{F})^{(2^{n-S}, 1^{2S-n-1})} + \sum_{j=1}^{n-S}[b^{2S-n}_{2j}+(2S-n+2(j-1)+1)b^{2S-n}_{2(j-1)}](P_{F})^{(2^{n-S-j}, 1^{2S-n+2j-1})}\\
			&= b_{0}^{2S-n} (P_{F})^{(2^{n-S}, 1^{2S-n-1})} + \sum_{j=1}^{n-S}b^{2S-n-1}_{2j}(P_{F})^{(2^{n-S-j}, 1^{2S-n+2j-1})}\\
			&=\sum_{j=0}^{(n-1)-(S-1)}b^{2(S-1)-(n-1)}_{2j}(P_{F})^{(2^{(n-1)-(S-1)-j}, 1^{2(S-1)-(n-1)+2j})},
		\end{aligned}
		\]
		which directly implies our claim. If $2S=n$, then the same type of argument can be used with the equality mentioned after Equation \ref{computemusefulbcoef} and the fact that  $(P^{(2^{j})})_{F}= (P_{F})^{(2^{j-1}, 1)}$.	The case $n=3$ follows from Theorem \ref{genStreitbergvectorspace}. \end{proof}
	
	By an argument similar to that in Lemma \ref{computem}, we can show that $P^{(2^{j}, 1^{l})}(\mathds{X}_{n})= \binom{n}{l}\frac{(2j)!}{2^{j}j!}$.
	
	\section{Probability interactions related to characteristic functions and multivariate Gaussians}\label{MultivariateGaussian}
	
	In this section we connect zero probability interaction with functional equations involving characteristic functions. For $F\subseteq\{1,\ldots,n\}$, the vector $(t_{F},\vec0)$ agrees with $t$ on $F$ and is zero on $F^{c}$.
	
	\begin{lemma}\label{gaussiansum}
		Let  $T \subset \Pi_{n}$ be a lower set and let $P \in \mathfrak{M}(\mathbb{R}^{n})$ be a probability measure with Fourier transform $\hat{P} =f$. Then  $T[P ]=0$ if and only if the function $f$ satisfies the functional equation
		\[
		\sum_{\pi \in T} \mu(\pi, \hat{\mathbf{1}}_{n})f_{\pi}(t)=0 , \quad t \in \mathbb{R}^{n}
		\]
		where $f_{\pi}(t)= \prod_{i=1 }^{|\pi|}f(t_{F_{i}}, \vec{0})$.
	\end{lemma}
	
	\begin{proof} Indeed, if $\pi= (F_{1}, \ldots, F_{|\pi|})$, then
		\[
		\int_{\mathbb{R}^{n}} e^{-i\langle x, \xi \rangle } dP_{\pi}(\xi) = \prod_{i=1}^{|\pi|} \int_{\mathbb{R}^{|F_{i}|}} e^{-i\langle x_{F_{i}}, \xi_{F_{i}} \rangle } dP_{F_{i}}(\xi_{F_{i}}) = \prod_{i=1}^{|\pi|} \int_{\mathbb{R}^{n}} e^{-i\langle (x_{F_{i}}, \vec{0}), \xi \rangle } dP(\xi).
		\]
		Also, by uniqueness of the Fourier transform of finite measures a measure $\eta \in \mathfrak{M}(\mathbb{R}^{n})$ is the zero measure if and only if $\hat{\eta}$ is the zero function.
	\end{proof}
	
	A  function $f:\mathbb{R}^{n} \to \mathbb{C}$ is called positive definite (PD) if for every finite collection of distinct points $x_{1}, \ldots, x_{m} \in  \mathbb{R}^{n} $  and scalars $c_{1}, \ldots, c_{m} \in  \mathbb{C} $ we have
	\[
	\sum_{i,j=1}^{m}c_{i}\overline{c_{j}}f(x_{i} - x_{j}) \geq 0.
	\]
	Bochner's theorem, see \cite{Bochner1933}, states that if $f$ is continuous and $f(0)=1$ then the function $f$ is PD if and only if $f=\hat{P}$ for some probability  $P \in \mathfrak{M}(\mathbb{R}^{n})$.
	
	By Lemma \ref{gaussiansum} and Bochner's theorem, for a continuous positive definite function $f: \mathbb{R}^{n} \to \mathbb{C}$ with $f(0)=1$, we define the operator
	\[
	T[f](t):= \sum_{\pi \in T} \mu(\pi, \hat{\mathbf{1}}_{n})f_{\pi}(t),
	\]
	which is then the zero function if and only if  $T[P]=0$. Note that  $P=P_{\pi}$ for some $\pi \in \Pi_{n}\setminus\{\hat{\mathbf{1}}_{n}\}$ if and only if the function $f$ is what is usually called separable with respect to $\pi$, that is, $f = f_{\pi}$.

	Our aim in this section is to provide subsets of the space of positive definite functions on $\mathbb{R}^{n}$ for which we can characterize when $T[f]$ vanishes.
	
	We recall that a function  $G:\mathbb{R}^{n} \to \mathbb{C}$ is called conditionally negative definite (CND) if $G(-x)=\overline{G(x)}$ and for every finite collection of distinct points $x_{1}, \ldots, x_{m} \in  \mathbb{R}^{n} $  and scalars $c_{1}, \ldots, c_{m} \in  \mathbb{C} $, satisfying $\sum_{i=1}^{m}c_{i}=0$, we have
	\[
	\sum_{i,j=1}^{m}c_{i}\overline{c_{j}}G(x_{i} - x_{j}) \leq 0.
	\]
	It is known that if  $G: \mathbb{R}^{n} \to \mathbb{R}$ is a continuous CND function for which $G(0)=0$, the function $f(x)= e^{-G(x)}$ is a continuous PD function with $f(0)=1$, see Theorem 2.2 on page 75 of \cite{Berg1984}.
	
	For $\pi=(F_{1},\ldots,F_{|\pi|})$, write $G_{\pi}(x):=\sum_{i=1}^{|\pi|}G(x_{F_{i}},\vec0)$. Thus, when $f=e^{-G}$, we have $f_{\pi}=e^{-G_{\pi}}$.
	
	\begin{lemma}\label{mainlevy} Let  $G: \mathbb{R}^{n} \to \mathbb{R}$ be a continuous CND function for which $G(0)=0$ and let $T \subset \Pi_{n}$.  If for every    $\pi \in  T\setminus\{\hat{\mathbf{1}}_{n}\}$
		\[
		\lim_{\min_{1\leq i \leq n }|x_{i}| \to \infty  }G(x) -G_{\pi}(x)=-\infty,
		\]
		then it holds that $T[e^{-G}]$ is not identically zero.
	\end{lemma}
	
	\begin{proof} Indeed, let $f(x)= e^{-G(x)}$. Then,
		\[
		T[f](t)=\sum_{\pi \in T} \mu(\pi, \hat{\mathbf{1}}_{n})f_{\pi}(t)= \sum_{\pi \in T} \mu(\pi, \hat{\mathbf{1}}_{n})e^{-G_{\pi}(t)}.
		\]
		Note that by the hypothesis on $G$, for every non-top partition $\pi\in T$, $e^{-G_{\pi}(x) + G(x)} \to 0$ when $\min_{1\leq i \leq n }|x_{i}| \to \infty$. Hence, if $T[f]$ is identically zero, by multiplying it by $e^{G(x)}$ and letting $\min_{1\leq i \leq n }|x_{i}| \to \infty$ we get that  $\mu( \hat{\mathbf{1}}_{n}, \hat{\mathbf{1}}_{n})=0 $, which is a contradiction.
	\end{proof}
	
	To obtain examples using Lemma \ref{mainlevy}, consider the continuous functions $g:[0,\infty)\to\mathbb R$ with $g(0)=0$ such that $g(\|x-y\|^{2})$ is CND on every Euclidean space. These are precisely the Bernstein functions normalized by $g(0)=0$ \cite{schoenbradial,Schilling2012}. Each has a unique representation
	\begin{equation}\label{bernsteinrepresentation}
		g(t)=at+\int_{0}^{\infty}(1-e^{-rt})\frac{1+r}{r}\,d\sigma(r),
	\end{equation}
	where $a\geq0$ and $\sigma$ is a finite nonnegative measure on $(0,\infty)$. Set $G(x)=g(\|x\|^{2})$ and $d\nu(r)=(1+r)r^{-1}d\sigma(r)$. For a proper partition $\pi$, let $q=|\pi|$ and $t_{i}=\|x_{F_{i}}\|^{2}$. The integrand
	\[
	h_{r}(x):=\sum_{i=1}^{q}(1-e^{-rt_{i}})-(1-e^{-r\sum_{i} t_{i}})
	\]
	satisfies $0\leq h_{r}(x)\leq q-1$ and tends to $q-1$ when $\min_{j}|x_{j}|\to\infty$. Consequently, dominated convergence when $\nu((0,\infty))<\infty$, and Fatou's lemma otherwise, give
	\[
	\lim_{\min_{j}|x_{j}|\to\infty}[G(x)-G_{\pi}(x)]=(1-|\pi|)\int_{0}^{\infty}\frac{1+r}{r}d\sigma(r)=(1-|\pi|)\lim_{t\to\infty}[g(t)-at].
	\]
	The last equality follows by monotone convergence in Equation \ref{bernsteinrepresentation}. This gives the following characterization.
	
	\begin{theorem}\label{radialcndinfinity}
		Let $g$ be a Bernstein function with $g(0)=0$, and let $a$ be the coefficient in Equation \ref{bernsteinrepresentation}. Suppose either $g(t)=at$ for every $t\geq0$, or $\lim_{t\to\infty}[g(t)-at]=\infty$. For every lower set $T\subseteq\Pi_{n}$, the function $T[e^{-g(\|\cdot\|^{2})}]$ is identically zero if and only if $g(t)=at$.
	\end{theorem}
	
	\begin{proof} If $g(t)=at$, then
		\[
		G_{\pi}(x) = \sum_{i=1}^{|\pi|}g(\sum_{j \in F_{i}}x_{j}^{2}) = \sum_{i=1}^{|\pi|}a(\sum_{j \in F_{i}}x_{j}^{2})= a\|x\|^{2}= g(\|x\|^{2}).
		\]
		Hence $T[e^{-g(\|\cdot \|^{2})}](x)= e^{-g(\|x\|^{2})}\sum_{\pi \in T} \mu(\pi, \hat{\mathbf{1}}_{n})=0$  by Equation \ref{soma01}.  Conversely, if $\lim_{t \to \infty }(g(t)-at)=\infty$, Lemma \ref{mainlevy} and the preceding computation prove our claim.
	\end{proof}
	
	Examples of Bernstein functions that satisfy the condition of Theorem  \ref{radialcndinfinity} include $t^{\alpha}$, for $0<\alpha<1$ and $\log(t+1)$.
	
	For another example, the Gaussian probability measure $N(\omega,\Sigma)$ on $\mathbb R^{n}$ has Fourier transform $e^{-G_{\omega,\Sigma}(x)}$, where $G_{\omega,\Sigma}(x)=i\langle x,\omega\rangle+\langle\Sigma x,x\rangle/2$. For a partition $\pi$, define $\Sigma_{\pi}$ by keeping the entries of $\Sigma$ within each block of $\pi$ and setting the entries between different blocks to zero. Then $N(\omega,\Sigma)_{\pi}=N(\omega,\Sigma_{\pi})$.
	
	The identifiability theorem of \cite{10.1214/aoms/1177698520} motivates the following result.
	\begin{theorem}
		For every lower set $T\subseteq\Pi_{n}$, $T[N(\omega,\Sigma)]$ is the zero measure if and only if $\Sigma=\Sigma_{\pi}$ for some $\pi\in T\setminus\{\hat{\mathbf1}_{n}\}$.
	\end{theorem}
	\begin{proof}
		If $\Sigma=\Sigma_{\pi}$ for such a partition, Lemma \ref{iducedzero} gives the result. Conversely, suppose $\Sigma\neq\Sigma_{\pi}$ for every proper $\pi\in T$. Group the terms in $T[N(\omega,\Sigma)]$ according to their distinct covariance matrices. The coefficient of $N(\omega,\Sigma)$ remains $\mu(\hat{\mathbf{1}}_{n}, \hat{\mathbf{1}}_{n})=1$, which is an absurd by the identifiability of the Gaussians.\end{proof}
	
    \bibliographystyle{siam} 
	\bibliography{References}       

@article{Bakirov2006,
  author = {Bakirov, Nail K. and Rizzo, Maria L. and Sz{\'e}kely, G{\'a}bor J.},
  year = {2006},
  title = {A multivariate nonparametric test of independence},
  journal = {Journal of Multivariate Analysis},
  volume = {97},
  pages = {1742--1756},
  doi = {10.1016/j.jmva.2005.10.005},
}

@book{Berg1984,
  author = {Berg, Christian and Christensen, Jens and Ressel, Paul},
  year = {1984},
  title = {Harmonic Analysis on Semigroups: Theory of Positive Definite and Related Functions},
  series = {Graduate Texts in Mathematics},
  volume = {100},
  publisher = {Springer},
}

@article{Bochner1933,
  author = {Bochner, S.},
  year = {1933},
  title = {Monotone {Funktionen}, {Stieltjessche Integrale} und harmonische {Analyse}},
  journal = {Mathematische Annalen},
  volume = {108},
  pages = {378--410},
}

@article{Boettcher2018,
  author = {B{\"o}ttcher, Bj{\"o}rn and Keller-Ressel, Martin and Schilling, Ren{\'e} L.},
  year = {2018},
  title = {Detecting independence of random vectors: Generalized distance covariance and {Gaussian} covariance},
  journal = {Modern Stochastics: Theory and Applications},
  volume = {5},
  pages = {353--383},
  doi = {10.15559/18-VMSTA116},
}

@article{Boettcher2019,
  author = {B{\"o}ttcher, Bj{\"o}rn and Keller-Ressel, Martin and Schilling, Ren{\'e} L.},
  year = {2019},
  title = {Distance multivariance: New dependence measures for random vectors},
  journal = {Annals of Statistics},
  volume = {47},
  pages = {2757--2789},
  doi = {10.1214/18-AOS1764},
}

@article{Chakraborty2019,
  author = {Chakraborty, Shubhadeep and Zhang, Xianyang},
  year = {2019},
  title = {Distance Metrics for Measuring Joint Dependence with Application to Causal Inference},
  journal = {Journal of the American Statistical Association},
  volume = {114},
  pages = {1638--1650},
  doi = {10.1080/01621459.2018.1513364},
}

@inproceedings{Gretton2005,
  author = {Gretton, Arthur and Bousquet, Olivier and Smola, Alex and Sch{\"o}lkopf, Bernhard},
  year = {2005},
  title = {Measuring statistical dependence with {Hilbert--Schmidt} norms},
  booktitle = {Algorithmic Learning Theory},
  pages = {63--77},
  publisher = {Springer},
}

@inproceedings{Gretton2008,
  author = {Gretton, Arthur and Fukumizu, Kenji and Teo, Choon H and Song, Le and Sch{\"o}lkopf, Bernhard and Smola, Alex J},
  year = {2008},
  title = {A kernel statistical test of independence},
  booktitle = {Advances in Neural Information Processing Systems},
  pages = {585--592},
  volume = {20},
  publisher = {Curran Associates, Inc.},
}

@article{Guella2025a,
  author = {Guella, Jean Carlo},
  year = {2025},
  title = {{Hilbert} space embeddings of independence tests of several variables with radial basis functions},
  journal = {Analysis and Applications},
  pages = {1--55},
  doi = {10.1142/S0219530526500089},
}

@article{Janson2021,
  author = {Janson, Svante},
  year = {2021},
  title = {On distance covariance in metric and {Hilbert} spaces},
  journal = {ALEA. Latin American Journal of Probability and Mathematical Statistics},
  volume = {18},
  pages = {1353--1393},
  doi = {10.30757/alea.v18-50},
}

@book{lancaster1969chi,
  author = {Lancaster, Henry Oliver},
  year = {1969},
  title = {The Chi-Squared Distribution},
  publisher = {John Wiley \& Sons},
}

@inproceedings{pmlr-v258-liu25f,
  author = {Liu, Zhaolu and Barahona, Mauricio and Peach, Robert},
  year = {2025},
  title = {Information-Theoretic Measures on Lattices for Higher-Order Interactions},
  booktitle = {Proceedings of The 28th International Conference on Artificial Intelligence and Statistics},
  series = {Proceedings of Machine Learning Research},
  volume = {258},
  pages = {2206--2214},
  publisher = {PMLR},
  editor = {Li, Yingzhen and Mandt, Stephan and Agrawal, Shipra and Khan, Emtiyaz},
}

@inproceedings{NEURIPS2023_74f11936,
  author = {Liu, Zhaolu and Peach, Robert and Mediano, Pedro A. M and Barahona, Mauricio},
  year = {2023},
  title = {Interaction Measures, Partition Lattices and Kernel Tests for High-Order Interactions},
  booktitle = {Advances in Neural Information Processing Systems},
  volume = {36},
  pages = {36991--37012},
  publisher = {Curran Associates, Inc.},
  editor = {Oh, A. and Naumann, T. and Globerson, A. and Saenko, K. and Hardt, M. and Levine, S.},
}

@book{Schilling2012,
  author = {Schilling, Ren{\'e} L. and Song, Renming and Vondra{\v c}ek, Zoran},
  year = {2012},
  title = {{Bernstein} Functions: Theory and Applications},
  volume = {37},
  publisher = {Walter de Gruyter},
  series = {De Gruyter Studies in Mathematics},
  edition = {2},
  doi = {10.1515/9783110269338},
}

@article{schoenbradial,
  author = {Schoenberg, Isaac J.},
  year = {1938},
  title = {Metric spaces and completely monotone functions},
  journal = {Annals of Mathematics},
  volume = {39},
  pages = {811--841},
  doi = {10.2307/1968466},
}

@inproceedings{NIPS2013_076a0c97,
  author = {Sejdinovic, Dino and Gretton, Arthur and Bergsma, Wicher},
  year = {2013},
  title = {A Kernel Test for Three-Variable Interactions},
  booktitle = {Advances in Neural Information Processing Systems},
  volume = {26},
  pages = {1124--1132},
  publisher = {Curran Associates, Inc.},
  editor = {Burges, C. J. and Bottou, L. and Welling, M. and Ghahramani, Z. and Weinberger, K. Q.},
}

@book{Stanley_2011,
  author = {Stanley, Richard P.},
  year = {2011},
  title = {Enumerative Combinatorics, Volume 1},
  series = {Cambridge Studies in Advanced Mathematics},
  publisher = {Cambridge University Press},
  edition = {2},
  doi = {10.1017/CBO9781139058520},
}

@article{streitberg1990lancaster,
  author = {Streitberg, Bernd},
  year = {1990},
  title = {{Lancaster} interactions revisited},
  journal = {The Annals of Statistics},
  pages = {1878--1885},
  volume = {18},
  number = {4},
  doi = {10.1214/aos/1176347885},
}

@article{Szekely2007,
  author = {Sz{\'e}kely, G{\'a}bor J. and Rizzo, Maria L. and Bakirov, Nail K.},
  year = {2007},
  title = {{Measuring and testing dependence by correlation of distances}},
  journal = {The Annals of Statistics},
  volume = {35},
  pages = {2769--2794},
  doi = {10.1214/009053607000000505},
}

@article{Szekely2009,
  author = {{Sz{\'e}kely}, G{\'a}bor J. and {Rizzo}, Maria L.},
  year = {2009},
  title = {{Brownian} distance covariance},
  journal = {The Annals of Applied Statistics},
  volume = {3},
  pages = {1236--1265},
  doi = {10.1214/09-AOAS312},
}

@article{10.1214/aoms/1177698520,
  author = {Yakowitz, Sidney J. and Spragins, John D.},
  year = {1968},
  title = {{On the Identifiability of Finite Mixtures}},
  journal = {The Annals of Mathematical Statistics},
  volume = {39},
  pages = {209--214},
  doi = {10.1214/aoms/1177698520},
  number = {1},
}

@article{Pfister2018,
  author = {Pfister, Niklas and B{\"u}hlmann, Peter and Sch{\"o}lkopf, Bernhard and Peters, Jonas},
  title = {Kernel-based tests for joint independence},
  journal = {Journal of the Royal Statistical Society: Series B (Statistical Methodology)},
  year = {2018},
  volume = {80},
  number = {1},
  pages = {5--31},
  doi = {10.1111/rssb.12235},
}
	
	\end{document}